\documentclass[12pt]{article}

\title{Interpolation and sampling sequences\\
       for mixed-norm spaces\thanks{Sections~1--3 of this paper are
       taken from the Ph.D. dissertation of the first author \cite{PKN}
       under the direction of the second author. Section~4 on sampling was
       completed later.}}

\author{Phuc K. Nguyen \and Daniel H. Luecking}

\date{January 23, 2018}

\usepackage{amsmath,amsfonts,amsthm}
\usepackage{mathabx}
\usepackage[pdftex,pdfpagemode=UseNone,%
    pdfstartview={XYZ null null null},%
    plainpages=false,hypertexnames=true]{hyperref}

\advance\topmargin -.5in
\advance\textheight 1in
\advance\oddsidemargin -.5in
\advance\textwidth 1in
\makeatletter
\renewcommand\section{\@startsection {section}{1}{\z@}%
                                   {-3.5ex \@plus -1ex \@minus -.2ex}%
                                   {2.3ex \@plus.2ex}%
                                   {\normalfont\bfseries}}
\renewcommand\subsection{\@startsection{subsection}{2}{\z@}%
                                     {-3.25ex\@plus -1ex \@minus -.2ex}%
                                     {1.5ex \@plus .2ex}%
                                     {\normalfont\scshape}}
\makeatother

\let\tu\textup

\newtheorem{theorem}{Theorem}
\newtheorem{lemma}[theorem]{Lemma}
\newtheorem{proposition}[theorem]{Proposition}
\newtheorem{corollary}[theorem]{Corollary}
\theoremstyle{remark}
\newtheorem{remark}{Remark}

\newcommand*{\Z}{\mathcal{Z}}
\newcommand*{\K}{\mathcal{K}}
\newcommand*{\ID}{\mathbb{D}}
\newcommand*{\IC}{\mathbb{C}}
\newcommand*{\eps}{\epsilon}
\newcommand*\norm[1]{\left\lVert#1\right\rVert}
\newcommand*\triplenorm[1]{\left\vvvert #1\right\vvvert}

\newcommand*\W{\mathcal{W}}
\newcommand*\U{\mathcal{U}}
\newcommand*\M{\mathcal{M}}
\newcommand*\C{\mathcal{C}}
\newcommand*\clos[1]{\mkern4.5mu\overline{\mkern-4.5mu #1}}
\newcommand*\defeq{\mathrel{\mathop:}=}

\begin{document}

\maketitle

\begin{abstract}
This paper extends the known characterization of interpolation and
sampling sequences for Bergman spaces to the mixed-norm spaces. The
Bergman spaces have conformal invariance properties not shared by the
mixed-norm spaces. As a result, different techniques of proof were
required.
\end{abstract}

\section{Introduction}
\subsection{The mixed-norm spaces}

For a function $f$ analytic in the unit disc $\ID =\{z\in\IC :\:
|z|<1\}$, the integral means are defined by
\begin{equation}\label{def:integralmean}
  M_p(r,f) =
  \left[
    \frac{1}{2\pi} \int_0^{2\pi} |f(re^{i\theta})|^p \,d\theta
  \right]^{1/p},\quad 0<p<\infty,
\end{equation}
and
\begin{equation*}
  M_{\infty}(r,f) = \max_{\theta} |f(re^{i\theta})|.
\end{equation*}

For $0<p<\infty,\, 0<q<\infty$, the mixed-norm space $A(p,q)$ is the set
of functions $f$ analytic in $\ID $ with
\begin{equation}\label{def:mixednorm}
 \begin{split}
 \norm{f}_{L(p,q)}
    &= \left[ \int_0^1 M_p(r,f)^q 2r\,dr
        \right]^{1/q} \\
    &= \left[ \int_0^1
          \left( \frac{1}{2\pi}
            \int_0^{2\pi} |f(re^{i\theta})|^p \,d\theta
          \right)^{q/p} 2r\,dr
        \right]^{1/q} <\infty.
 \end{split}
\end{equation}

If $f\in A(p,q)$, we write $\norm{f}_{A(p,q)}$ for $\norm{f}_{L(p,q)}$.

If $p=q$, $A(p,q)$ is the Bergman space $A^p$:
\begin{equation}\label{def:bergmanspace}
  A^p =
    \left\{ f \text{ analytic in } \ID :
    \norm{f}_{A^p} = \left( \int_{\ID } |f(z)|^p \,dA(z)
    \right)^{1/p} < \infty \right\}.
\end{equation}

For $0<p,q<\infty$, $A(p,q)$ are invariant complete metric spaces with
the metric
\begin{equation}
  d(f,g) = \norm{f-g}_{A(p,q)}^s, \text{  where } s=\min(p,q,1).
\end{equation}
If $1\le p$ and $1\le q$, $\norm{.}_{A(p,q)}$ is a norm and
$(A(p,q),\norm{.})$ becomes a Banach space.

We are also interested in two other spaces which are related to the
mixed-norm space. The first one is the growth space $A^{-n}\, (n>0)$,
which is the set of functions $f$ analytic in $\ID $ with
\begin{equation}\label{def:growthspace}
  \norm{ f}_{-n}
    =
    \sup_{z\in\ID } \,(1-|z|^2)^n |f(z)| < \infty.
\end{equation}
The second one is the weighted Bergman space $A_\alpha^p
\,(0<p<\infty,\alpha>-1)$, which consists of functions $f$
analytic in $\ID $ with
\begin{equation}\label{def:weightedbergmanspace}
  \norm{ f }_{p,\alpha} = \left\{
    \int_{\ID } |f(z)|^p (1-|z|^2)^{\alpha} dA(z)
    \right\}^{1/p} < \infty.
\end{equation}

Note that $A^p = A^p_{0}$. Also note that some authors use that
convention that $A^p_\alpha$ is equal to what we would call $A^p_{\alpha
p - 1}$.

\subsection{Definitions of interpolation sequences}\label{sec:definitions}

Let $A$ be a space of functions on $\Omega$, $X$ a sequence space and
$\Gamma \equiv (z_m) \subset \Omega$ a sequence that has no limit points
in $\Omega$. Denote by $R_\Gamma$ the mapping $f \mapsto (f(z_m))$. We say
$\Gamma$ is an interpolation sequence for $(A,X)$ if $R_\Gamma(A)=X$. In other
words, $\Gamma$ is an interpolation sequence for $(A,X)$ if $(f(z_m))
\in X$ for all $f \in A$ and for every sequence $(a_m) \in X$, there is
a function $f \in A$ such that $f(z_m)=a_m$ for every $m$.

If $A=A(p,q)$, we let the sequence space $X$ to be $l^{p,q}$ defined as
folows:

Let $\beta = 1/L$ for some integer $L \geq 2$ and set $r_j = 1-\beta^j,
\,j=0,1,2,\dotsc$. Let us divide the unit disc $\ID $ into annuli
$A_j = \{ z\in\IC :\: r_j \le |z| < r_{j+1} \}$. Also divide each
annulus $A_j$ by means of equally spaced radii into $2L^j$ equal
`polar rectangles' $ Q_{j,k} = \{ z=r^{i\theta}: \,r_j \leq r < r_{j+1},
\,(k-1)\beta^j\pi \leq \theta < k\beta^j\pi \} $.

We now arrange $\Gamma$ such that
\begin{equation*}
  |z_1| \leq |z_2| \leq \dotsc < 1.
\end{equation*}
For each annulus $A_j$, let $L_j$ be the number of points of $\Gamma$ in
$A_j$ (necessarily finite). Number the points of $\Gamma$ by $z_{j,k}$
such that $z_{j,k} \in A_j, \,k=1,2,\dotsc,L_j$ and $|z_{j,k}| < |z_{j',k'}|$
if $j<j'$. Specifically, $(z_{j,k})$ is a doubly indexed
sequence defined by
\begin{equation}\label{eqn:doubleindex}
  z_{j,k} = z_m \text{ if } m = k + \sum_{i=0}^{j} L_i.
\end{equation}
Unless specified otherwise, every doubly indexed sequence from now on is
numbered according to \eqref{eqn:doubleindex}.

Now let $(a_m)$ be a sequence in $\IC $ and let $(a_{j,k})$ be its
doubly indexed sequence defined in the same manner, i.e. $a_{j,k} = a_m
\text{ if } m = k + \sum_{i=0}^{j} L_i$. Let us define a sequence space
$l^{p,q}(\Gamma)$ to consist of all sequence $(a_m)$ with
\begin{equation}
  \norm{ (a_m) }_{l^{p,q}(\Gamma)}
  \equiv
  \norm{ (a_{j,k}) }_{l^{p,q}(\Gamma)} = \left\{
    \sum_{j=0}^\infty (1-r_j)^{1+q/p}
    \left( \sum_{k=1}^{L_j} |a_{j,k}|^p \right)^{q/p}
    \right\}^{1/q} < \infty.
\end{equation}
The sequence space $l^{p,q}$ is now used for interpolation for $A(p,q)$.
We say $\Gamma$ is an interpolation sequence for $A(p,q)$ if
$R_\Gamma (A(p,q)) = l^{p,q}$.

We also require the definitions of interpolation sequences for $A^{-n}$
and $A_\alpha^p$. The corresponding sequence spaces are
$l_{-n}^\infty(\Gamma)$ and $l_\alpha^p(\Gamma)$ defined by
\begin{align}
  l_{-n}^\infty(\Gamma)
    &= \left\{
        (a_m) \subset \IC :\:
        \norm{ (a_m) }_{-n,\Gamma}
        =
        \sup_m (1-|z_m|^2)^n |a_m| < \infty
        \right\}, \\
  l_\alpha^p(\Gamma)
    &= \left\{
        (a_m) \subset \IC :\: \norm{ (a_m) }_{p,\alpha,\Gamma}^p
        =
        \sum_m |a_m|^p (1-|z_m|^2)^{\alpha + 2} < \infty
        \right\}.
\end{align}
Thus we say $\Gamma$ is an interpolation sequence for $A^{-n}$ if
$R_\Gamma (A^{-n})=l_{-n}^\infty(\Gamma)$ and $\Gamma$ is an interpolation
sequence for $A_\alpha^p$ if $R_\Gamma (A_\alpha^p)=l_\alpha^p(\Gamma)$.

\subsection{Uniformly Discrete Sequences}\label{sec:uniformlydiscrete}

Recall that for $0<p,q<\infty$, $A(p,q)$ and $l^{p,q}$ are invariant
complete metric spaces with the metric $d(x,y) = \norm{x-y}^s$,
$s=\min(p,q,1)$. Let $\Gamma=(z_m)$ be an interpolation sequence for
$A(p,q)$. A simple verification shows that the mapping $R_\Gamma : f \mapsto
(f(z_m))$ has closed graph and hence is bounded from $A(p,q)$ into
$l^{p,q}$. Since $\Gamma$ is an interpolation sequence, $R_\Gamma$ is also
onto. The open mapping theorem implies that there exists the smallest
constant $M(\Gamma)$ such that for every $(a_m) \in l^{p,q}$, there is
an $f \in A(p,q)$ satisfying $f(z_m)=a_m$ for all $m$ and
$\norm{f}_{A(p,q)} \leq M\norm{(a_m)}_{l^{p,q}}$. $M$ is called the
interpolation constant of $\Gamma$ for  $A(p,q)$.

Interpolation sequences are not too dense anywhere. In fact, a necessary
condition for interpolation is being uniformly discrete. A sequence
$\Gamma \equiv (z_m) \subset \ID $ is said to be uniformly
discrete with separation constant $\delta$ if
\begin{equation}
  \delta := \inf\{ \rho(z_n,z_m): n \neq m \} > 0
\end{equation}
where $\rho(z,w)$ is the pseudohyperbolic distance between two points in
$\ID $ given by
\begin{equation}
  \rho(z,w) = \dfrac{|z-w|}{|1-\bar{w}z|}.
\end{equation}
We also denote the pseudohyperbolic disk of radius $r$ centered at $z$
by
\begin{equation}
  E(z,r) = \{ \zeta\in\ID  : \rho(z,\zeta) < r \}.
\end{equation}

The following results on uniformly discrete sequences were proved in
\cite[Chapter 2]{durenschuster}:

\begin{proposition}\label{prop:uniformlydiscrete}
Suppose that $\Gamma=(z_m)$ is a uniformly discrete sequence with
$\rho(z_n,z_m) \geq \delta > 0$ for $m \neq n$. Let $n(\Gamma,z,r)$
denote the number of points in $\Gamma \cap E(z,r)$. Then
\begin{enumerate}
  \item ${\displaystyle\sum_{m=1}^\infty} (1-|z_m|^2)^2 \leq \dfrac{4}{\delta^2}$.
  \item $n(\Gamma,z,r) \leq \left(\frac{2}{\delta}+1\right)^2
        \dfrac{1}{1-r^2}$ for every point $z \in \ID $ and $0<r<1$. In
        particular, $n(\Gamma,z,r) = O\left(\dfrac{1}{1-r}\right)$.
\end{enumerate}
\end{proposition}

\subsection{Seip's theorems and their extensions}

In \cite{seip}, Seip characterizes sets of sampling and interpolation
for $A^{-n}$ via certain densities which are equivalent to the
Korenblum's densities in \cite{Korenblum}. In order to state the
results, we require several definitions.

For $0<s<1$, let $n(\Gamma,\zeta,s)$ be the number of points of $\Gamma$
contained in $E(\zeta,s)$. Define, for $\Gamma$ uniformly discrete,
\begin{equation}
  F(\Gamma,\zeta,r) =
  \frac{{\displaystyle \int_0^r n(\Gamma,\zeta,s) \,ds}}
  {{\displaystyle 2\int_0^r a(E(0,s)) \,ds}}
\end{equation}
where
\begin{equation*}
  a(\Omega) = \int_\Omega \frac{1}{(1-|z|^2)^2} \,dA(z)
\end{equation*}
is the hyperbolic measure of a measurable subset $\Omega$ of $\ID $.

The lower and upper uniform densities are defined, respectively, to be
\begin{equation}
  D^-(\Gamma) = \liminf_{r \rightarrow 1} \inf_{\zeta \in \ID }
  F(\Gamma,\zeta,r)
\end{equation}
and
\begin{equation}
  D^+(\Gamma) = \limsup_{r \rightarrow 1} \sup_{\zeta \in \ID }
  F(\Gamma,\zeta,r).
\end{equation}
The densities can be reformulated as
\begin{align}
  D^-(\Gamma)
    &= \liminf_{r\rightarrow 1} \inf_{\zeta \in\ \ID }
        \dfrac{\displaystyle \sum\limits_{1/2 < \rho(\zeta,z_n) < r}
        \log\dfrac{1}{\rho(\zeta,z_n)}}
        {\left(\log \dfrac{1}{1-r}\right)}, \\
  D^+(\Gamma) &= \limsup_{r\rightarrow 1} \sup_{\zeta \in\ \ID }
        \dfrac{\displaystyle \sum\limits_{1/2 < \rho(\zeta,z_n) < r}
        \log\dfrac{1}{\rho(\zeta,z_n)}}
        {\left(\log \dfrac{1}{1-r}\right)}.
\end{align}

The following theorems were proved in \cite{seip}. (The definition of
sampling sequence occurs in section~\ref{sec:sampling}.)

\begin{theorem}\label{seiptheoremsampling}
A sequence $\Gamma$ of distinct points in $\ID $ is a set of
sampling for $A^{-n}$ if and only if it contains a uniformly discrete
subsequence $\Gamma'$ for which $D^-(\Gamma') > n$.
\end{theorem}

\begin{theorem}\label{seiptheoreminterpolation}
A sequence $\Gamma$ of distinct points in $\ID $ is a set of
interpolation for $A^{-n}$ if and only if $\Gamma$ is uniformly discrete
and $D^+(\Gamma) < n$.
\end{theorem}

Theorem \ref{seiptheoreminterpolation} was extended to the Bergman space
$A^p \,(0<p<\infty)$ and the weighted Bergman space $A_\alpha^p
\,(0<p<\infty, \,\alpha>-1)$. The proofs can be found in
\cite{durenschuster}, \cite{schuster} and \cite{jevtic}.

\begin{theorem}\label{weightedbergmansampling}
A sequence $\Gamma$ of distinct points in $\ID $ is a set of
sampling for $A_\alpha^p$ if and only if it contains a uniformly
discrete subsequence $\Gamma'$ for which $D^-(\Gamma') > (1+\alpha)/p$.
\end{theorem}

\begin{theorem}\label{weightedbergmaninterpolation}
A sequence $\Gamma$ of distinct points in $\ID $ is a set of
interpolation for $A_\alpha^p$ if and only if $\Gamma$ is uniformly
discrete and $D^+(\Gamma) < (1 + \alpha)/p$.
\end{theorem}

\section{Some properties of interpolation sequences for mixed-norm
spaces}

\subsection{Basic properties of mixed norm
spaces}\label{sec:basicproperties}

Recall that the mixed norm space $A(p,q)$ consists of analytic functions
$f$ in the unit disk with
\begin{equation*}
  \norm{ f}_{A(p,q)} = \left[ \int_0^1 \left( \frac{1}{2\pi}
  \int_0^{2\pi} |f(re^{i\theta})|^p \,d\theta \right)^{q/p}
  2r\,dr \right]^{1/q} < \infty.
\end{equation*}
It is easy to check that $A(p,q)$ is an invariant metric space with the
metric
\begin{equation}\label{def:mixednormmatrix}
    d(f,g):=\norm{f-g}_{A(p,q)}^s
\end{equation}
where $s=\min(p,q,1)$. If $p,q > 1$, $(A(p,q),\norm{.}_{A(p,q)})$ is a
normed linear space.

The following results can be found in \cite{gadbois} and \cite{benedek}:

\begin{proposition}\label{prop:basicproperties} {\ }
  \begin{enumerate}
    \item Let $0<p,q<\infty$ and $f \in A(p,q)$, then for all $z \in \ID $:
       \begin{equation}
         |f(z)| \leq C\norm{f}_{A(p,q)} (1-|z|)^{-(1/p+1/q)}
       \end{equation}
       for some constant $C$ independent of $f$.
    \item If $d$ is the metric defined in \eqref{def:mixednormmatrix},
          then $(A(p,q),d)$ is an invariant, complete metric space and
          $(A(p,q),\norm{.}_{A(p,q)})$ is a Banach space if $p,q\ge1$.
    \item If $p,q > 1$ and $1/p+1/p'=1,\, 1/q+1/q'=1$, then there is a
          continuous linear isomorphism of $A(p',q')$ onto the dual space of
          $A(p,q)$.
  \end{enumerate}
\end{proposition}

\subsection{Discrete versions of mixed norms}\label{sec:discrete}

Let $f$ be analytic in $\ID $. Then $|f|^p$ is subharmonic for any
$p>0$. This implies that the integral means
\begin{equation*}
  M_p(f,r)=\left(\dfrac{1}{2\pi}\int_{-\pi}^\pi |f(re^{i\theta})|^p
  \,d\theta\right)^{1/p}
\end{equation*}
are increasing function of $r$. Let us adopt the notations
$A_n, r_n, \beta, L, L_n, Q_{n,k}$ used in Section~\ref{sec:definitions}. The
norm $\norm{ f}^q_{A(p,q)}$ can then be replaced with a summation
\begin{equation*}
  \norm{ f}^q_{A(p,q)} =\sum_{n=1}^\infty \int_{r_{n-1}}^{r_n}
  M_p(f,r)^q \,2r\,dr.
\end{equation*}
If $r_{n-1} < r < r_n < r' < r_{n+1}$, then
\begin{equation*}
  M_p(f,r)^p \leq M_p(f,r')^p.
\end{equation*}
Integrate this inequality with respect to $2r'\,dr'$ from $r_n$
to $r_{n+1}$ to get
\begin{equation*}
  (r_{n+1}^2 - r_n^2)M_p(f,r)^p \leq \int_{r_n}^{r_{n+1}} M_p(f,r')^q
  \,2r'\,dr' \leq \dfrac{1}{\pi}\int_{A_{n+1}} |f(z)|^p
  \,dA(z).
\end{equation*}
This gives
\begin{equation*}
  M_p(f,r)^p \leq \dfrac{1}{|A_{n+1}|}\int_{A_{n+1}} |f(z)|^p
  \,dA(z)
\end{equation*}
where absolute bars indicate the area of a set. We can now raise both
sides to the power $q/p$ and integrate with respect to $2r\,dr$
from $r_{n-1}$ to $r_n$:
\begin{equation*}
  \begin{split}
    \int_{r_{n-1}}^{r_n} M_p(f,r)^q \,2r\,dr
    &\leq (r_n^2 - r_{n-1}^2) \left(\dfrac{1}{|A_{n+1}|}\int_{A_{n+1}}
        |f(z)|^p \,dA(z)\right)^{q/p} \\
    &\leq 2\beta^{n-1} \left(\dfrac{1}{|A_{n+1}|}\int_{A_{n+1}} |f(z)|^p
        \,dA(z)\right)^{q/p}.
  \end{split}
\end{equation*}
A similar argument but integrating first in $r$ and then $r'$ gives
\begin{equation*}
 \begin{split}
    \int_{r_n}^{r_{n+1}} M_p(f,r)^q \,2r\,dr
    &\geq (r_{n+1}^2 - r_n^2) \left(\dfrac{1}{|A_n|}\int_{A_n} |f(z)|^p
        \,dA(z)\right)^{q/p} \\
    &\geq \beta^n(1-\beta) \left(\dfrac{1}{|A_n|}\int_{A_n} |f(z)|^p
        \,dA(z)\right)^{q/p}.
  \end{split}
\end{equation*}
Summing both these inequalities gives us (for some constant C depending
only on $\beta$)
\begin{multline*}
  \sum_{n=2}^\infty \dfrac{1-r_n}{C} \left(\dfrac{1}{|A_n|}\int_{A_n}
    |f(z)|^p \,dA(z)\right)^{q/p}
  \leq \norm{ f}^q_{A(p,q)} \\
  \leq \sum_{n=1}^\infty C(1-r_n) \left(\dfrac{1}{|A_n|}\int_{A_n}
    |f(z)|^p \,dA(z)\right)^{q/p}.
\end{multline*}
Since $M_p(f,r)$ is increasing, the integral of $|f|^p$ over $A_1$ is
less than a constant times the integral over $A_2$ and so we can include
$n=1$ in the sum on the left, provided we increase the constant $C$.
Thus we obtain the following new norm
\begin{equation}
  \triplenorm{f}^q_{A(p,q)} =
    \sum_{n=1}^\infty (1-r_n)
    \left(\dfrac{1}{|A_n|}\int_{A_n} |f(z)|^p \,dA(z)\right)^{q/p}.
\end{equation}
which is equivalent to the usual norm $\norm{ f}_{A(p,q)}$ (that is
$\triplenorm{f} \leq C\norm{f}$ and $\norm{f} \leq C\triplenorm{f}$ for
some constant $C$ depending only on $\beta$).

The integral over the annulus $A_n$ can be written as a sum (on $k$) of
integrals over $Q_{n,k}$:
\begin{equation*}
  \triplenorm{f}^q_{A(p,q)}
    =
    \sum_{n=1}^\infty (1-r_n) \left(\dfrac{1}{|A_n|} \sum_{k=1}^{L_n}
    \int_{Q_{n,k}} |f(z)|^p \,dA(z)\right)^{q/p}.
\end{equation*}
Or, since $|A_n|=2L^n|Q_{n,k}|$:
\begin{equation*}
  \triplenorm{f}^q_{A(p,q)}
    =
    \sum_{n=1}^\infty (1-r_n) \left(\dfrac{L^{-n}}{2} \sum_{k=1}^{L_n}
    \dfrac{1}{|Q_{n,k}|} \int_{Q_{n,k}} |f(z)|^p
    \,dA(z)\right)^{q/p}.
\end{equation*}
Or, since $L^{-n}=\beta^n=1-r_n$:
\begin{equation}\label{eqn:discretenorm}
  \triplenorm{f}^q_{A(p,q)}
    =
    \sum_{n=1}^\infty (1-r_n) \left(\dfrac{1-r_n}{2} \sum_{k=1}^{L_n}
    \dfrac{1}{|Q_{n,k}|} \int_{Q_{n,k}} |f(z)|^p
    \,dA(z)\right)^{q/p}.
\end{equation}
We will need to a few properties of the set $Q_{n,k}$, especially in
connection with the pseudohyperbolic metric $\rho$, defined by
\begin{equation*}
  \rho(z,w) \equiv \dfrac{|z-w|}{|1-\bar{w}z|} \;\; \text{(for
  $z,w\in\ID $)}
\end{equation*}
and the associated hyperbolic disks defined by
\begin{align*}
  E(z,R) &\equiv \{w:\rho(z,w)<R\} \\
  \clos{E}(z,R) &\equiv \{w:\rho(z,w)\leq R\}
\end{align*}

\begin{theorem}
There exist constants $0<r<R<1$ depending only on $\beta$ such that
\begin{enumerate}
  \item For every pair $(n,k)$, if $z \in Q_{n,k}$ then $Q_{n,k} \subset
    \clos{E}(z,R)$
  \item For every pair $(n,k)$, there exists $z \in Q_{n,k}$ with $E(z,r)
    \subset Q_{n,k}$.
\end{enumerate}
\end{theorem}

Let us denote by $r_\beta$ the largest possible $r$ and $R_\beta$ the
smallest possible $R$ from this theorem. Let $\tilde{z}_{n,k}$ denote
the $z$ from statement $2$ for $r=r_\beta$ and let $z_{n,k}$ denote the
Euclidean center of $E_{n,k} \equiv E(\tilde{z}_{n,k},r_\beta)$.

A fairly easy estimate of the areas involved show that there are
constants $C_1$ and $C_2$ depending only on $\beta$ such that for all
pairs $(n,k)$
\begin{equation*}
  |E_{n,k}| \leq |Q_{n,k}| \leq C_1|E_{n,k}|
\end{equation*}
and for any $z \in Q_{n,k}$
\begin{equation*}
  |Q_{n,k}| \leq |E(z,R_\beta)| \leq C_2|Q_{n,k}|.
\end{equation*}
Then we have the estimate, for every analytic function $f$:
\begin{equation*} |f(z_{n,k})|^p
  \leq \dfrac{1}{|E_{n,k}|} \int_{E_{n,k}} |f|^p \,dA
  \leq \dfrac{C_1}{|Q_{n,k}|} \int_{Q_{n,k}} |f|^p \,dA.
\end{equation*}
Note also that for any $z \in A_n$ we have
\begin{equation}\label{ineq:noname01}
  \frac{1}{C} \leq \frac{1-|z|^2}{1-r_n} \leq C
\end{equation}
for some constant $C$ depending only on $\beta$. And so for any $f$ in $A(p,q)$,
\begin{equation}
  \left(
  \sum_{n=1}^\infty (1-r_n)
  \left( \sum_{k=1}^{L_n} (1-|z_{n,k}|^2)|f(z_{n,k})|^p \right)^{q/p}
  \right)^{1/q}
  \leq C\triplenorm{f}_{A(p,q)}
  \leq C \norm{ f }_{A(p,q)}
\end{equation}
for some constant $C=C(\beta)$. This show that evaluation at the points
of the sequence $(z_{n,k})$ is a bounded map from $A(p,q)$ into
$l^{p,q}((z_{n,k}))$. This result is also true for a uniformly discrete
sequence $\Gamma = (z_m) \equiv (z_{n,k})$ ($(z_{n,k})$ is defined as in
Section~\ref{sec:definitions}). Proposition \ref{prop:uniformlydiscrete}
implies that there is an upper bound $M$ on the number of points of
$\Gamma$ in any $Q_{n,k}$. So, if $\Gamma$ is uniformly discrete then it
is the union of at most $M$ sequences each of which has at most one
point in any $Q_{n,k}$.

\begin{theorem}\label{thm:discretebounded}
If $\Gamma = \{ z_m \in \ID \} = \{z_{n,k} \in \ID : n \geq
1, 1 \leq k \leq L_n \}$ is uniformly discrete then the operator
$R_\Gamma$ taking $f$ to $(f(z_{n,k}))$ is bounded from $A(p,q)$ to
$l^{p,q}(\Gamma)$.
\end{theorem}

\begin{proof}
By the remarks preceding the theorem, we can suppose that there is at
most one point of $\Gamma$ in each $Q_{n,k}$. If $z_{n,k} \in \Gamma
\cap Q_{n,k}$, let it be the Euclidean center of a disk $E_{n,k}$ with
hyperbolic radius $r_\beta$. $E_{n,k}$ needs not be contained in
$Q_{n,k}$, but if it intersects the boundary of $Q_{n,k}$ then it is
contained in the union of those $Q_{n',k'}$ that are adjacent to the
part of the boundary intersected. This union contains at most one
additional $Q_{n,k'}$ in the same annulus $A_n$ as $Q_{n,k}$ and at most
$L+1$ additional $Q_{n+1,k'}$ in the next annulus $A_{n+1}$ (or at most
2 additional $Q_{n-1,k'}$ in the previous annulus $A_{n-1}$). Thus
\begin{equation*}
  |f(z_{n,k})|^p
  \leq \dfrac{1}{|E_{n,k}|} \int_{E_{n,k}} |f|^p \,dA
  \leq C \sum_{Q_{n',k'} \cap D_{n,k} \neq \emptyset}
    \dfrac{1}{|Q_{n',k'}|} \int_{Q_{n',k'}} |f|^p \,dA
\end{equation*}
It is now easy to see from \ref{eqn:discretenorm} that $\lVert f
\rVert_{l^{p,q}(\Gamma)}$ is less than a finite multiple of
$\triplenorm{f}_{A(p,q)}$.
\end{proof}

\subsection{Necessity of separation}\label{sec:separation}

One might wonder why $l^{p,q}$ is chosen to be the target space for
interpolation. It turns out that relatively mild assumptions on $A$ and
$X$ force an interpolation sequence $\Gamma$ to be uniformly discrete.
Moreover, Theorem \ref{thm:discretebounded} shows that if $\Gamma \equiv
(z_k)$ is uniformly discrete then $R_\Gamma(A(p,q)) \subset l^{p,q}$
with $R_\Gamma(f)
= (f(z_k))$. This suggests we choose $l^{p,q}$ for the interpolation
problem. We now turn to the proofs of these assertions.

Let $\Gamma = \{z_k:k=1,2,3,\ldots\} \subset \ID $. Let us define
$M_a(z)=(a-z)/(1-\bar{a}z)$. Let $e^{(k)}$ denote the sequence having a
$1$ in position $k$ and $0$ elsewhere. Let $P_k$ be the operator of
projection onto the $k$th component: if $w=(w_j)$ then $P_k(w)=w_k
e^{(k)}$.

\begin{theorem}\label{necessityofseparation}
Let $A$ be a Banach space of analytic functions on $\ID $ and
$\Gamma$ a sequence of distinct points in $\ID $. Let X be a
Banach space of sequences, the sequences being indexed the same as
$\Gamma$. Let $R_\Gamma$ be the operator that takes functions to sequences via
$R_\Gamma(f)_k = f(z_k)$. Assume the following:
\begin{enumerate}
  \item For every $k$, the sequence $e^{(k)}$ belongs to $X$.
  \item For every $k$, $P_k$ takes $X$ continuously into $X$ and $\sup_k
    \norm{ P_k } < \infty$.
  \item There is a constant $C_A$ such that for any $k$, if $f\in A$ and
    $f(z_k)=0$, then $f/M_{z_k} \in A$ and $\norm{ f/M_{z_k} }_A \leq
    C_A\norm{ f }_A$.
  \item Covergence in $A$ implies pointwise convergence on $\Gamma$.
\end{enumerate}
If the operator $R_\Gamma$ satisfies $R_\Gamma(A)=X$, then $\Gamma$ is uniformly discrete.
\end{theorem}

\begin{proof}
Condition 2 and 4 imply that $R_\Gamma$ has closed graph and so is bounded.
Since it is also onto we can apply the open mapping principle to obtain
an interpolation constant $K$: every sequence $w\in X$ is $R_\Gamma f$ for some
$f\in A$ with $\norm{ f}_A \leq K\norm{ w}_X$.

Let $z_k,z_n\in\Gamma, \; k\neq n$. The assumptions imply that the
sequence $w=M_{z_n}(z_k) e^{(k)}$ belongs to $X$. Let $f\in A$ with
$\norm{ f}_A \leq K\norm{ w}_X$. Since $f$ vanishes at $z_n$
we have $g=f/M_{z_n}\in A$. Note that $g(z_k)=1$, and so
$P_k(R_\Gamma(g))=e^{(k)}$. Then we have the following inequalities:
\begin{equation*}
  \norm{ e^{(k)}}
    \leq C_1\norm{ R_\Gamma(g)}
    \leq C_2\norm{ g}
    \leq C_3\norm{ f}
    \leq C_4\norm{ w}
    = C_4|M_{z_n}(z_k)|\norm{ e^{(k)}},
 \end{equation*}
where $C_1=\sup_k\norm{ P_k}, \;C_2=C_1\norm{ R_\Gamma}, \;C_3=C_A
C_2,$ and $C_4=KC_3$. We immediately obtain
$\rho(z_k,z_n)=|M_{z_n}(z_k)|\geq 1/C_4$.
\end{proof}

Certainly, $l^{p,q}(\Gamma)$ satisfies the requirements. It is not
immediately obvious that $A(p,q)$ satisfies \text{condition 3}. We
address that in the next result.

\begin{corollary}\label{cor:necessityofseparation}
Interpolation sequences for $A(p,q)$ are uniformly discrete.
\end{corollary}

\begin{proof}
We will show that $A(p,q)$ satisfies \text{condition 3} of the previous
theorem. Let $f$ vanish at a fixed $z_k$. Consider the disk
$D=D(z_k,1/2)=\{z:|M_{z_k}(z)<1/2\}$ and let $g=f/M_{z_k}$. On the
complement of this disk it is clear that $|g|<2|f|$. It will be enough
to show that $\norm{ g}_{A(p,q)} \leq \lVert g \chi_{\ID \backslash D}
\rVert_{L(p,q)}$ for any analytic function g, where the constant $C$
does not depend on $z_k$. For this it suffices to show that
\begin{equation*}
  \norm{ g\chi_{D}} \leq \norm{ g\chi_{\ID \backslash D}}.
\end{equation*}
The proof of this is essentially a sort of maximum principle. Its proof
would be a great deal easier if we were able to use a conformal map to
turn integral centered around $z_k$ to integrals centered around $0$.

We consider the equivalent norm $\triplenorm{.}$ on $A(p,q)$ obtained in
Section~\ref{sec:discrete}. Let $D'$ be the slightly larger disk with
pseudohyperbolic radius $\dfrac{1/2+1/2}{1+(1/2)(1/2)}=4/5$. If we
choose the parameter $\beta$ appropriately, we can arrange for $D'$ to
overlap at most two annuli $A_n$. For simplicity, let us temporarily
assume that $D'$ is contained in exactly one annulus $A_n$. We did not
extend the new norm to all of $L^{p,q}$, but it is clear that we can
apply it to any measurable function, though it needs not always be
finite. In the case $g\chi_D$ we get a single nonzero term in the
infinite sum:
\begin{equation*}
  \triplenorm{g\chi_D}^q_{A(p,q)}
    =
    (1-r_n)\left(\dfrac{1}{|A_n|}\int_D |g|^p \,dA \right)^{q/p}.
\end{equation*}
It is relatively straightforward to show that $\int_D |g|^p
\,dA \leq C\int_{D'\backslash D} |g|^p \,dA$
(conformally map to the disks of radius 1/2 and 4/5 centered at 0, use polar
coordinates and the fact that the integral on circles increases as the
radius increases). This gives us
\begin{equation*}
\begin{split}
 \triplenorm{g\chi_D}^q_{A(p,q)}
    &\leq C(1-r_n)\left(\dfrac{1}{|A_n|}\int_{D'\backslash D} |g|^p
        \,dA \right)^{q/p} \\
    &\leq C(1-r_n)\left(\dfrac{1}{|A_n|}\int_{A_n\backslash D} |g|^p
        \,dA \right)^{q/p} \\
    &\leq C \sum_n (1-r_n)\left(\dfrac{1}{|A_n|}\int_{A_n}
        |g\chi_{\ID \backslash D}|^p \,dA \right)^{q/p} \\
    &= C\triplenorm{g\chi_{\ID \backslash D}}^q_{A(p,q)}.
\end{split}
\end{equation*}
If $D'$ overlaps two annuli $A_{n-1}$ and $A_n$, then the second line
can be replaced by a sum of two terms.

So we established that $A(p,q)$ satisfies the conditions of Theorem
\ref{necessityofseparation}. Now if $\Gamma$ is an interpolation
sequence for $A(p,q)$, i.e. $R_\Gamma(A(p,q))=l^{p,q}(\Gamma)$, then $\Gamma$
is uniformly discrete thanks to Theorem \ref{necessityofseparation}.
\end{proof}

\subsection{Stability under perturbation}\label{sec:stability}

A property of interpolation sequences is their stability under
(hyperbolically) small perturbations. We start with the following lemma:

\begin{lemma}\label{lem:sta}
Let $\Gamma = (z_m)$ and $\Gamma' = (z'_m)$ be two sequences in
$\ID $ with no limits in $\ID $ such that $\rho(z_m,z'_m) <
\delta$ for all $m$. If $\delta$ is sufficiently small; then for every
sequence $(a_m)$ in $l^{p,q}(\Gamma)$, $(a_m)$ also belongs in
$l^{p,q}(\Gamma')$.
\end{lemma}

\begin{proof}
First, for $a,b,c>0$ and $\alpha > 0$, a simple application of Holder's
and Minkowski's inequalities gives us
\begin{equation}\label{ine:lem:sta}
  (a+b+c)^\alpha \leq C_\alpha (a^\alpha + b^\alpha + c^\alpha)
\end{equation}
where $C_\alpha = \max\{ 1,3^{\alpha-1} \}$.

Second, since $z'_m \in E(z_m,\delta)$, $|z_m - z'_m| < 2R$ where
\begin{equation*}
    R = \dfrac{\delta (1-|z_m|^2)}{1-\delta^2|z_m|^2}
\end{equation*}
is the Euclidean radius of the hyperbolic disk $E(z_m,\delta)$. This gives us
\begin{equation}
  |z_m - z'_m| < \dfrac{2\delta (1-|z_m|^2)}{1-\delta^2|z_m|^2} <
    \dfrac{4\delta}{1-\delta^2}(1-|z_m|).
\end{equation}
Thus for $\delta < 1/20$:
\begin{equation*}
  |z_m - z'_m| < \dfrac{4\delta}{1-\delta^2}(1-|z_m|) < \dfrac{1}{4}(1-|z_m|).
\end{equation*}
Now let $(z_{j,k})$ be the doubly indexed sequence of $(z_m)$ defined as
in \eqref{eqn:doubleindex}. Since $r_j < |z_{j,k}| < r_{j+1}$ and $1-r_j
= \beta^j$, we have
\begin{equation*}
  |z_m - z'_m| < \dfrac{1}{4}(1-r_j) < \dfrac{\beta(r_j-r_{j-1})}{4(1-\beta)}.
\end{equation*}
It follows that $z'_m \in A_{j-1} \cup A_j \cup A_{j+1}$. This and
inequality \eqref{ine:lem:sta} imply
\begin{equation}
  \norm{ (a_m) }^q_{l^{p,q}(\Gamma')}
    \leq 3\beta^{-(1+q/p)}C_{q/p}\norm{(a_m)}^q_{l^{p,q}(\Gamma)}
    < \infty.
\end{equation}
Hence, $(a_m)$ is also a sequence in $l^{p,q}(\Gamma')$.
\end{proof}

The following theorem shows that if $\Gamma$ is an interpolation
sequence for $A(p,q)$ then a small perturbation on $\Gamma$ still results
in an interpolation sequence for $A(p,q)$. The proof is taken after
Lemma 1.9 in \cite{jevtic} and Theorem 5.1 in \cite{luecking1}, save a
few minor changes to work for mixed-norm spaces.

\begin{theorem}\label{thm:stability}
For $0<p,q<\infty$, let $\Gamma = (u_m)$ be an interpolation sequence
for $A(p,q)$ and $(u'_m)$ be another sequence in $\ID $. There
exists $\delta > 0$ such that if
\begin{equation*}
  \rho(u_m,u'_m) < \delta \quad \text{for all } m
\end{equation*}
then $\Gamma' = (u'_m)$ is also an interpolation sequence for $A(p,q)$.
\end{theorem}

\begin{proof}
Let $(v_m) \in l^{p,q}(\Gamma')$. By Lemma \ref{lem:sta}, $(v_m) \in
l^{p,q}(\Gamma)$ for $\delta$ small enough. Denote $u=(u_m), u'=(u'_m)$
and $v^0=v$. Since $u$ is an interpolation sequence of $A(p,q)$, there exists
$f_0 \in A(p,q)$ such that $f_0(u)=v^0$ (i.e., $f_0(u_m)=v_m$ for all $m$).
Suppose $v^1 \in l^{p,q}(\Gamma)$, take now $v^1 := v^0-f_0(u'), f_1 \in
A(p,q)$ with $f_1(u)=v^1$, and define $v^2=v^1-f_1(u')$. An iteration of
this construction provides functions $f_m \in A(p,q)$ with
$f_0(u')+f_1(u')+\dots+f_{m-1}(u')+f_m(u) = v^0 = v$. If we can prove
that there exists $0<\gamma<1$ such that $\norm{ v_{m+1} } \leq
\gamma\norm{ v_m }$ for $m=0,1,2,\dots$ then
\begin{equation*}
  \norm{f_m} \leq M\norm{v^m} \leq M\gamma^m\norm{v^0},
\end{equation*}
where $M$ is the interpolation constant of $(u_m)$. The interpolation
problem for $(u'_m)$ is then solved by the function $f=\sum_m f_m$. To
this end, we use a general estimate for analytic function which can be
found in \cite{jevtic} or \cite{luecking1}:
\begin{equation}\label{ine:thm:sta}
  |f(z)-f(w)|^p \leq
    C\rho^p(z,w) \int\limits_{E(w,r)} |f(\zeta)|^p (1-|\zeta|^2)^{-2}
    \,dA(z)
\end{equation}
where $C=C(p,r)>0$ and $r\geq 2\rho(z,w)$. Provided that $\delta$ is chosen
small enough so that the hyperbolic disks $E(z_m,r)$ are pairwise
disjoint, we have:
\begin{align*}
  \norm{v^1}^q_{l^{p,q}(\Gamma)}
    &= \sum_{j} (1-r_j)^{1+q/p} \left( \sum_k
        |f_0(u_{j,k})-f_0(u'_{j,k})|^p \right)^{q/p}  \\
    &\leq \sum_{j} (1-r_j)^{1+q/p}
        \left( \sum_k C_1\delta^p
        \int\limits_{E(u_{j,k},r)} |f_0(\zeta)|^p (1-|\zeta|^2)^{-2}
        \,dA(\zeta) \right)^{q/p} \\
    &\leq C_1^{q/p}\delta^q \sum_j (1-r_j)^{1+q/p}
        \left( \sum_k
        \int\limits_{E(u_{j,k},r)} |f_0(\zeta)|^p (1-|\zeta|^2)^{-2}
        \,dA(\zeta) \right)^{q/p} \\
    &\leq C_2\delta^q \sum_j (1-r_j)^{1+q/p}
        \left( \sum_k \int\limits_{E(u_{j,k},r)}
        \dfrac{1}{|E(u_{j,k},r)|} |f_0(\zeta)|^p \,dA(\zeta)
        \right)^{q/p}
\end{align*}
The same argument as in the proof of Theorem \ref{thm:discretebounded} shows that
\begin{equation*}
  \int\limits_{E(u_{j,k},r)} \dfrac{1}{|E(u_{j,k},r)|} |f_0(\zeta)|^p
    \,dA(\zeta)
  \leq C_3(r,\beta) \sum_{Q_{j',k'} \cap E_{j,k} \neq \emptyset}
    \dfrac{1}{|Q_{j',k'}|} \int_{Q_{j\,',k'}} |f|^p \,dA
\end{equation*}
where the number of $Q_{j\prime,k'} \cap E_{j,k} \neq \emptyset$ is at
most $L+4$. Hence,
\begin{align*}
  \norm{v^1}^q
    &\leq C_4\delta^q \sum_j (1-r_j)^{1+q/p}
        \left( \sum_k \dfrac{1}{|Q_{j,k}|} \int_{Q_{j,k}} |f|^p
        \,dA \right)^{q/p} \\
    &\leq C_5\delta^q \norm{f_0}^q \\
    &\leq C_5\delta^p M^q \norm{v^0}^q,
\end{align*}
where $C_5$ depends only on $p,q,\delta$ and $\beta$. Let $\gamma^q =
C_5\delta^q M^q$ and choose $\delta$ small enough so that $\gamma < 1$.
The same estimate shows that $\norm{ v_{m+1} } \leq \gamma\lVert
v_m \rVert$ for $m=0,1,2,\dots$. This completes the proof.
\end{proof}

We see that interpolation sequences are stable under small perturbations.
On the other hand, a small perturbation can increase the upper uniform
density of a sequence. The following lemma was proved in
\cite{domanski}:

\begin{lemma}\label{lem:domanski}
Let $\Gamma=(z_n)$ be a uniformly discrete sequence in $\ID $ with
$\rho(z_m,z_n)>\beta$ for $m\neq n$ and $0<\delta<\beta<1/2$. If
$D^+(\Gamma)<\infty$, then there is a sequence $\Gamma' = (z'_n) \subset
\ID $ with $\rho(z_n,z'_n) \leq \delta$ for all $n$, such that
\begin{equation*}
  D^+(\Gamma') \geq (1+\delta)D^+(\Gamma)
\end{equation*}
\end{lemma}

\begin{remark}\label{rem:stability}
Suppose that $D^+(\Gamma)\leq \gamma$ for every interpolation sequence
$\Gamma$ for $A(p,q)$. Then by Theorem~\ref{thm:stability} and Lemma
\ref{lem:domanski}, there exists an interpolation sequence $\Gamma'$ for
$A(p,q)$ such that $D^+(\Gamma')>D^+(\Gamma)$. Since
$D^+(\Gamma')\leq\gamma$, we must have $D^+(\Gamma)<\gamma$. Thus in
order to prove that $D^+(\Gamma)<\gamma$ for every interpolation
sequence $\Gamma$ for $A(p,q)$, we only need to show that
$D^+(\Gamma)\leq\gamma$ for all $\Gamma$.
\end{remark}

\section{Interpolation sequences for mixed-norm spaces}\label{sec:main}

We propose to prove
\begin{theorem} \label{interpolationthm}
A sequence $\Gamma$ of distinct points in $\ID $ is a set of
interpolation for $A(p,q)$ if and only if $\Gamma$ is uniformly discrete
and $D^+(\Gamma) < 1/q$.
\end{theorem}
The proof is split into several cases and requires the following inequalities

\begin{lemma}\label{lem:inequalities} {\ }
\begin{enumerate}
  \item For $0<p\le 1$,
    \begin{equation}\label{ineq01}
    \left(
        \sum_{n=1}^\infty |b_n|
        \right)^p
        \le \sum_{n=1}^\infty |b_n|^p
    \end{equation}

  \item For $M>1$, there exists a constant $C=C(M)$ such that for all $0<\rho<1$
    \begin{equation}\label{ineq02}
      \frac{1/C} {(1-\rho)^{M-1}}
        \le
        \int_{-\pi}^{\pi}
        \frac{1} {|1-\rho e^{i\theta}|^M}
        \,d\theta
        \le
        \frac{C} {(1-\rho)^{M-1}}
    \end{equation}

  \item For $-1<a<B-1$, there exists a constant $C=C(a,B)$ such that for all $0<\rho<1$
    \begin{equation}\label{ineq03}
        \frac{1/C} {(1-\rho)^{B-a-1}}
        \le
        \int_{0}^{1}
        \frac{(1-r)^a} {(1-r\rho)^B}
        \,dr
        \le
        \frac{C}
        {(1-\rho)^{B-a-1}}
    \end{equation}

  \item For $-1<a<M-2$, there exists a constant $C=C(a,M)$ such that for
        all $w\in\ID $
    \begin{equation}\label{ineq04}
        \frac{1/C} {(1-|w|)^{M-a-2}}
        \le \int_{\ID }
        \frac{(1-|z|^2)^a} {|1-\bar{w}z|^M} \,dA(z)
        \le \frac{C} {(1-|w|)^{M-a-2}}
    \end{equation}

  \item Let $\Gamma = (z_k)$ be a uniformly discrete sequence in $\ID $
    with seperation constant $\delta = \delta(\Gamma)$. Then for $1<t<s$,
    there is a constant $C = C(t,s,\delta) > 0$ such that
    \begin{equation}\label{ineq05}
        \sum_{k=1}^\infty
        \frac{(1-|z_k|^2)^t} {|1-\bar{z}z_k|^s}
        \le C(1-|z|^2)^{t-s}
    \end{equation}
        for all $z\in\ID $.
\end{enumerate}
\end{lemma}
\begin{proof}
The techniques for the proof can be found in [Lemma 1-3, Chapter 6,
\cite{durenschuster}].
\end{proof}

\subsection{Suffiency part of Theorem~\ref{interpolationthm}}

Suppose $\Gamma \equiv (z_{mk})$ is uniformly discrete and
$D^+(\Gamma)<1/q$. Then $\Gamma$ is an interpolation sequence for
$A^{-n}$ where $n=1/q-\eps$ for some $\eps > 0$. Hence there exists a
sequence $(g_{mk}) \in A^{-n}$ such that $g_{mk}(z_{mk}) =
(1-|z_{mk}|^2)^{-n}$, $g_{mk}(z_{m'k'}) = 0$ for all $(m',k') \neq
(m,k)$, $\lVert g_{mk} \rVert \leq M(\Gamma)$ for all $m$ and $k$; where
$M(\Gamma)$ is the interpolation constant of $\Gamma$ for $A^{-n}$.

The interpolation problem is then solved by the formula
\begin{equation}\label{interpolationformula}
  f(z) = \sum_{m}\sum_{k} a_{mk}g_{mk}(z) \dfrac{(1-|z_{mk}|^2)^{n+s}}
    {(1-\bar{z}_{mk}z)^s}
\end{equation}
where $s$ is suffiently large.

It is clear that if the series converges point-wise to $f$ then
$f(z_{mk})=a_{mk}$ for all $(m,k)$. For each $0 < R < 1$, we will show
that the series converges uniformly on the disk $\{z: |z|\leq R \}$.
First, we see that
\begin{equation*}
\begin{split}
  \dfrac{|g_{mk}(z)|}{(1-\bar{z}_{mk}z)^s}
    & \leq (1-|z|^2)^{-n} \norm{g_{mk}(z)}_{A^{-n}} (1-R)^{-s} \\
    & \leq (1-R^2)^{-n} (1-R)^{-s} M = C(R,M(\Gamma)).
\end{split}
\end{equation*}
Thus, it suffices to show that
\begin{equation}\label{summationrequirement}
  \sum_{m}\sum_{k} |a_{mk}|(1-|z_{mk}|^2)^{n+s} < \infty.
\end{equation}
If $p \le 1$, then
\begin{equation}
  \sum_k |a_{mk}| \leq \left( \sum_k |a_{mk}|^p \right)^{1/p}
\end{equation}
and if $p > 1$ then, since the number of points $z_{mk}$ in the annulus
$A_m$ is $O(1/(1-r_m))$,
\begin{equation}
  \sum_k |a_{mk}|
    \leq
    \left(\frac{C}{(1-r_m)}\right)^{1/p'}\left( \sum_k |a_{mk}|^p \right)^{1/p}
\end{equation}
where $C$ depends only on $R$ and the separation constant of $\Gamma$.
We also have $1-|z_{mk}| \approx 1-r_m$. In either case the series in
\eqref{summationrequirement} is bounded by
\begin{equation}
  C\sum_m (1-r_m)^{n + s - 1} \left( \sum_k |a_{mk}|^p \right)^{1/p}.
\end{equation}
If $q \le 1$ this sum is bounded by
\begin{equation}
  C\left(\sum_m (1-r_m)^{(n + s - 1)q} \left( \sum_k
    |a_{mk}|^p \right)^{q/p} \right)^{1/q}.
\end{equation}
otherwise H\"older's inequality gives a bound of
\begin{equation}
  C\left(\sum_m (1-r_m)^{(n + s - 1)} \left( \sum_k
    |a_{mk}|^p \right)^{q/p} \right)^{1/q}\left( \sum_m (1-r_m)^{(n + s
    - 1)}\right)^{1/q'}.
\end{equation}
Now note that $\sum_m (1-r_m^2)$ is finite. Consequently, a
straightforward estimation shows that, for $s$ sufficiently large,
\begin{equation}
  \sum_m (1-r_m)^{n+s-1} \sum_k |a_{mk}| \leq C \norm{(a_{mk})}_{l^{p,q}}^q < \infty
\end{equation}
where $C$ depends only on $R,p,q,M(\Gamma)$ and the separation constant
of $\Gamma$. It follows that the series in \eqref{interpolationformula}
converges uniformly to an analytic function $f$ on each compact set of
$\ID $. To prove that $f$ belongs in $A(p,q)$, we only need to
verify that $\norm{f}_{A(p,q)} < \infty$. The proof is split into four
cases. The only property of the $g_{mk}$ that we will use is that for
all $m$ and $k$, $|g_{mk}(z)| \le M(\Gamma) (1 - r^2)^{-n}$. Thus, it
suffices to show that
\begin{equation}
    h(z) \defeq \sum_{m} \sum_{k} |a_{mk}| (1-|z|^2)^{-n}
        \frac{(1-|z_{mk}|^2)^{n+s}} {|1-\bar{z}_{mk}z|^s}
\end{equation}
belongs to $L(p,q)$

\paragraph{Case 1:} $q/p \leq 1, \,p>1$.\\
Let $x_1,x_2,y_1,y_2$ be real numbers (to be specified later) satisfying
$x_1+y_1=n+s, x_2+y_2=s$. By Holder's inequality, we have
\begin{multline*}
  \norm{ h }_{p,q}^{q}
     = \int_{0}^{1}
        \left\{ \int_{0}^{2\pi}
            \left[ \sum_{m} \sum_{k} |a_{mk}|(1-r^2)^{-n}
                \frac{(1-|z_{mk}|^2)^{n+s}} {|1-\bar{z}_{mk}z|^s}
            \right]^p \,d\theta
        \right\}^{q/p} 2r\,dr\\
    \le \int_{0}^{1} (1-r^2)^{-nq}
    \Biggl\{ \int_{0}^{2\pi}
        \left[ \sum_{m} \sum_{k} |a_{mk}|
            \frac{(1-|z_{mk}|^2)^{x_1 p}} {|1-\bar{z}_{mk}z|^{x_2 p}}
        \right] \times \\
    \times
        \left[ \sum_m \sum_k
            \frac{(1-|z_{mk}|^2)^{y_1 p'}} {|1-\bar{z}_{mk}z|^{y_2 p'}}
        \right]^{p/p'} \,d\theta
    \Biggl\}^{q/p} 2r\,dr.
\end{multline*}
Inequality \eqref{ineq05} then gives us
\begin{multline*}
  \norm{ h }_{p,q}^{q}
    \leq C\int_{0}^{1} (1-r^2)^{-nq}
        \Bigg\{ \int_{0}^{2\pi}
         \left[ \sum_m \sum_k |a_{mk}|
         \frac{(1-|z_{mk}|^2)^{x_1 p}} {|1-\bar{z}_{mk}z|^{x_2 p}}
         \right] \times \\
         \shoveright{ \times \left[ (1-r^2)^{(y_1-y_2)p'} \right]^{p/p'} \,d\theta
        \Bigg\}^{q/p} 2r\,dr } \\
  = C\int_{0}^{1} (1-r^2)^{-q(n-y_1+y_2)}
     \Bigg[ \sum_{m} \sum_{k} |a_{mk}|^p (1-|z_{mk}|^2)^{x_1 p}
      \int_{0}^{2\pi} \frac{1} {|1-\bar{z}_{mk}z|^{x_2 p}} \,d\theta
     \Bigg]^{q/p} 2r\,dr.
\end{multline*}
In view of inequality \eqref{ineq02}, we obtain
\begin{align*}
  \norm{ h }_{p,q}^{q}
    \le
    C \int_{0}^{1} (1-r^2)^{-q(n-y_1+y_2)}
    \left[
    \sum_{m} \sum_{k} |a_{mk}|^p
    \frac{(1-|z_{mk}|^2)^{x_1 p}} {(1-|z_{mk}|r)^{x_2 p-1}}
    \right]^{q/p}
    2r\,dr
\end{align*}
From inequalities \eqref{ineq01} and \eqref{ineq:noname01}, we have
\begin{align*}
  \norm{ h }_{p,q}^{q}
    & \le C \int_{0}^{1} (1-r^2)^{-q(n-y_1+y_2)}
        \sum_m
        \frac{(1-r_m^2)^{x_1 q}} {(1-r_m r)^{qx_2-q/p}}
        \left(
        \sum_k| a_{mk}|^p
        \right)^{q/p}
        2r\,dr \\
    & = C \sum_m (1-r_m^2)^{x_1 q}
        \int_0^1
        \frac{(1-r^2)^{-q(n-y_1+y_2)}} {(1-r_m r)^{qx_2-q/p}}\,
        2r\,dr
        \left(
        \sum_k|a_{mk}|^p
        \right)^{q/p}
\end{align*}
Finally, inequality \eqref{ineq03} implies
\begin{align*}
  \norm{ h }_{p,q}^{q}
    & \le
        C \sum_m (1-r_m^2)^{x_1 q} (1-r_m^2)^{-q(n-y_1+y_2+x_2-x_1-1/p)+1}
        \left(
        \sum_k|a_{mk}|^p
        \right)^{q/p} \\
    & =
        C \sum_m (1-r_m^2)^{1+q/p}
        \left(
        \sum_k| a_{mk}|^p
        \right)^{q/p} < \infty.
\end{align*}

To apply the inequalities above, $x_1,x_2,y_1,y_2$ must be chosen such that
\[\begin{cases}
1 < y_1 p' < y_2 p', \\
1 < x_2 p, \\
-1 < -q(n-y_1+y_2) < qx_2-q/p-1
\end{cases}\]
The following works for $\eps$ so small such that $n+\eps < 1/q$:
\[\begin{cases}
y_1 = 1/p' + \eps, \,y_2 = 1/p' + 2\eps \\
x_1 = n+s-y_1, \,x_2 = s-y_2 \\
s > 1-\eps
\end{cases}\]
We conclude that $f \in A(p,q)$.

\paragraph{Case 2:} $q/p \leq 1, \,p\le 1$.\\
From inequality \eqref{ineq01}, we obtain
\begin{align*}
  \norm{ h }_{p,q}^{q}
    &= \int_{0}^{1}
        \left\{
        \int_{0}^{2\pi}
        \left[
        \sum_{m} \sum_{k} |a_{mk}|(1-r^2)^{-n}
        \frac{(1-|z_{mk}|^2)^{n+s}} {|1-\bar{z}_{mk}z|^s}
        \right]^p
        \,d\theta
        \right\}^{q/p}
        2r\,dr\\
    &\le \int_{0}^{1}
        \left[
        \int_{0}^{2\pi} \sum_m \sum_k |a_{mk}|^p(1-r^2)^{-np}
        \frac{(1-|z_{mk}|^2)^{(n+s)p}} {|1-\bar{z}_{mk}z|^{sp}}\,
        \,d\theta
        \right]^{q/p} 2r\,dr \\
    &= \int_0^1
        \left[
        \sum_m\sum_k |a_{mk}|^p
        (1-r^2)^{-np} (1-|z_{mk}|^2)^{(n+s)p}
        \int_0^{2\pi} \frac{1} {|1-\bar{z}_{mk}z|^{sp}}
        \,d\theta
        \right]^{q/p} 2r\,dr
\end{align*}
In light of inequalities \eqref{ineq02} and \eqref{ineq:noname01}, this implies that
\begin{align*}
  \norm{ h }_{p,q}^{q}
    &\le C \int_0^1
        \left[
        \sum_m\sum_k |a_{mk}|^p
        (1-r^2)^{-np} (1-|z_{mk}|^2)^{(n+s)p}
        \frac{1} {(1-|z_{mk}|r)^{sp-1}}
        \right]^{q/p}
        2r\,dr \\
    &\le C \int_0^1
        \left[
        \sum_m (1-r^2)^{-np}
        \frac{(1-r_m^2)^{(n+s)p}} {(1-r_m r)^{sp-1}}
        \sum_k |a_{mk}|^p
        \right]^{q/p}
        2r\,dr,
\end{align*}
provided $s > 1/p$. By inequality \eqref{ineq01},
\begin{align*}
  \norm{ h }_{p,q}^{q}
    &\le C \int_0^1
        \sum_m (1-r^2)^{-nq}
        \frac{(1-r_m^2)^{(n+s)q}} {(1-r_m r)^{sq-q/p}}
        \left(
        \sum_k |a_{mk}|^p
        \right)^{q/p}
        2r\,dr \\
    &= C\sum_m (1-r_m^2)^{(n+s)q}
        \int_0^1
        \frac{(1-r^2)^{-nq}} {(1-r_m r)^{sq-q/p}}
        2r\,dr
        \left(
        \sum_k |a_{mk}|^p
        \right)^{q/p}
\end{align*}
By inequality \eqref{ineq03}, if $s > 1/p + 1/q - n$, we have
\begin{align*}
  \norm{ h }_{p,q}^{q}
    &\le C \sum_m (1-r_m^2)^{(n+s)q} (1-r_m^2)^{-nq-sq+q/p+1}
        \left(
        \sum_k |a_{mk}|^p
        \right)^{q/p} \\
    &= C \sum_m (1-r_m^2)^{1+q/p}
        \left(
        \sum_k |a_{mk}|^p
        \right)^{q/p} < \infty.
\end{align*}

\paragraph{Case 3:} $q/p > 1, \,p > 1$.\\
Let $x_1,y_1,x_2,y_2$ be real numbers such that $x_1+y_1=n+s,x_2+y_2=s$
and apply Holder's inequality to get
\begin{multline*}
  \norm{ h }_{p,q}^{q}
    = \int_{0}^{1}
        \left\{
        \int_{0}^{2\pi}
        \left[
        \sum_{m} \sum_{k} |a_{mk}|(1-r^2)^{-n}
        \frac{(1-|z_{mk}|^2)^{n+s}} {|1-\bar{z}_{mk}z|^s}
        \right]^p
        \,d\theta
        \right\}^{q/p}
        2r\,dr\\
    \le
    \int_{0}^{1} (1-r^2)^{-nq}
    \Bigg\{
    \int_{0}^{2\pi}
    \left[
    \sum_m\sum_k |a_{mk}|^p
    \frac{(1-|z_{mk}|^2)^{px_1}} {|1-\bar{z}_{mk}z|^{px_2}}
    \right] \times \\
    \times \left[
    \sum_m \sum_k
    \frac{(1-|z_{mk}|^2)^{p'y_1}} {|1-\bar{z}_{mk}z|^{p'y_2}}
    \right]^{p/p'}
    \,d\theta
    \Bigg\}^{q/p}
    2r\,dr
\end{multline*}
If $p'y_1<p'y_2$, we can apply inequality \eqref{ineq05} to get
\begin{multline*}
  \norm{ h }_{p,q}^{q}
    \le C \int_{0}^{1} (1-r^2)^{-nq}
    \left\{
    \int_{0}^{2\pi}
    \left[
    \sum_m\sum_k |a_{mk}|^p
    \frac{(1-|z_{mk}|^2)^{px_1}} {|1-\bar{z}_{mk}z|^{px_2}}
    \right]
    (1-r^2)^{(y_1-y_2)p}
    \,d\theta
    \right\}^{q/p}
    2r\,dr\\
    =
    C \int_{0}^{1} (1-r^2)^{(y_1-y_2-n)q}
    \left[
    \sum_m\sum_k |a_{mk}|^p
    (1-|z_{mk}|^2)^{px_1}
    \int_0^{2\pi}
    \frac{1} {|1-\bar{z}_{mk}z|^{px_2}}
    \,d\theta
    \right]^{q/p}
    2r\,dr
\end{multline*}
If $px_2>1$, inequality \eqref{ineq02} gives us
\begin{align*}
  \norm{ h }_{p,q}^{q}
    & \le C \int_{0}^{1} (1-r^2)^{(y_1-y_2-n)q}
        \left[
        \sum_m \sum_k |a_{mk}|^p
        (1-|z_{mk}|^2)^{px_1}
        \frac{1} {(1-|z_{mk}|r)^{px_2-1}}
        \right]^{q/p}
        2r\,dr
\end{align*}
Let $x_3,x_4,y_3,y_4$ be real numbers such that
$x_3+y_3=x_1,x_4+y_4=x_2-1/p$. By Holder's inequality,
\begin{multline*}
    \norm{ h }_{p,q}^{q}
    \le
    C \int_{0}^{1} (1-r^2)^{(y_1-y_2-n)q}
    \left[
    \sum_m
    \frac{(1-r_m^2)^{qx_3}} {(1-rr_m)^{qx_4}}
    \left(
    \sum_k |a_{mk}|^p\right)^{q/p}
    \right] \times \\
    \times \left[
    \sum_m
    \frac{(1-r_m^2)^{py_3(q/p)'}} {(1-rr_m)^{py_4(q/p)'}}
    \right]^{\frac{q}{p(q/p)'}}
    2r\,dr
\end{multline*}
If $y_3<y_4$, inequality \eqref{ineq05} gives us
\begin{align*}
  \norm{ h }_{p,q}^{q}
    &\le C \int_{0}^{1}
        (1-r^2)^{(y_1-y_2-n)q}
        \left[
        \sum_m
        \frac{(1-r_m^2)^{qx_3}} {(1-rr_m)^{qx_4}}
        \left(
        \sum_k |a_{mk}|^p
        \right)^{q/p}
        \right] \times{}\\
    &\qquad\quad {}\times (1-r^2)^{q(y_3-y_4)} 2r\,dr\\
    &= C \sum_m
        \left(
        \sum_k |a_{mk}|^p
        \right)^{q/p}
        (1-r_m^2)^{qx_3}
        \int_0^1
        \frac{(1-r^2)^{q(y_1-y_2+y_3-y_4-n)}} {(1-rr_m)^{qx_4}}\,
        2r\,dr
\end{align*}
If $-1<q(y_1-y_2+y_3-y_4-n)<qx_4 - 1$ then we can apply inequality \eqref{ineq03} to get
\begin{align*}
  \norm{ h }_{p,q}^{q}
    & \le C \sum_m
        \left(
        \sum_k |a_{mk}|^p
        \right)^{q/p}
        (1-r_m^2)^{qx_3}
        (1-r_m^2)^{q(y_1-y_2+y_3-y_4-x_4-n)+1} \\
    & = C \sum_m (1-r_m^2)^{1+q/p}
        \left(
        \sum_k |a_{mk}|^p
        \right)^{q/p} < \infty.
\end{align*}
It is easy to see that there exist $x_1,x_2,\dots$ satisfying the
conditions for the inequalities. Thus $f \in A(p,q)$.

\paragraph{Case 4:} $q/p > 1, \,p\leq 1$.\\
By inequalities \eqref{ineq01} and \eqref{ineq02}, we have
\begin{align*}
  \norm{ h }_{p,q}^{q}
    & = \int_{0}^{1}
        \left\{
        \int_{0}^{2\pi}
        \left[
        \sum_m \sum_k |a_{mk}| (1-r^2)^{-n}
        \frac{(1-|z_{mk}|^2)^{n+s}} {|1-\bar{z}_{mk}z|^s}
        \right]^p
        \,d\theta
        \right\}^{q/p}
        2r\,dr \\
    & \le \int_{0}^{1}
        \left\{
        \int_{0}^{2\pi}
        \sum_m \sum_k |a_{mk}|^p (1-r^2)^{-np}
        \frac{(1-|z_{mk}|^2)^{(n+s)p}} {|1-\bar{z}_{mk}z|^{sp}}\,
        d\theta
        \right\}^{q/p}
        2r\,dr \\
    & = \int_{0}^{1} (1-r^2)^{-nq}
        \left[
        \sum_m \sum_k |a_{mk}|^p
        (1-|z_{mk}|^2)^{(n+s)p}
        \int_{0}^{2\pi}
        \frac{1} {|1-\bar{z}_{mk}z|^{sp}}\,
        d\theta
        \right]^{q/p}
        2r\,dr \\
    &\le C \int_{0}^{1} (1-r^2)^{-nq}
        \left[
        \sum_m \sum_k |a_{mk}|^p
        (1-|z_{mk}|^2)^{(n+s)p}
        \frac{1} {(1-|z_{mk}|r)^{sp-1}}
        \right]^{q/p}
        2r\,dr
\end{align*}
Let $x_1,y_1,x_2,y_2$ be real numbers such that
$x_1+y_1=n+s,x_2+y_2=s-1/p$ and apply Holder's inequality to get
\begin{multline*}
    \norm{ h }_{p,q}^{q}
    \le
    C \int_{0}^{1} (1-r^2)^{-nq}
    \left[
    \sum_m
    \left(
    \sum_k |a_{mk}|^p
    \right)^{q/p}
    \frac{(1-r_m^2)^{qx_1}} {(1-rr_m)^{qx_2}}
    \right] \times \\
    \times \left[
    \sum_m
    \frac{(1-r_m^2)^{py_1(q/p)'}} {(1-rr_m)^{py_2(q/p)'}}
    \right]^{(q/p)/(q/p)'}
    2r\,dr
\end{multline*}
Inequalities \eqref{ineq05} and \eqref{ineq03} then give us
\begin{align*}
  \norm{ h }_{p,q}^{q}
    & \le C \int_{0}^{1} (1-r^2)^{-nq}
        \left[
        \sum_m
        \left(
        \sum_k |a_{mk}|^p
        \right)^{q/p}
        \frac{(1-r_m^2)^{qx_1}} {(1-rr_m)^{qx_2}}
        \right]
        (1-r^2)^{(y_1-y_2)q} 2r\,dr \\
    & = C \sum_m (1-r_m^2)^{qx_1}
        \int_0^1
        \frac{(1-r^2)^{q(y_1-y_2-n)}} {(1-rr_m)^{qx_2}}\,
        dr
        \left(
        \sum_k |a_{mk}|^p
        \right)^{q/p} \\
    & \le C \sum_m (1-r_m^2)^{qx_1} (1-r_m^2)^{q(y_1-y_2-x_2-n)+1}
        \left(
        \sum_k |a_{mk}|^p
        \right)^{q/p} \\
    & = C \sum_m (1-r_m^2)^{1+q/p}
        \left(
        \sum_k |a_{mk}|^p
        \right)^{q/p} < \infty
\end{align*}
The last step is choosing $x_1,x_2,y_1,y_2$ for applying the
inequalities. Having verified all the cases, we conclude that $f \in
A(p,q)$.

This completes the proof of the suffiency part of Theorem~\ref{interpolationthm}.

\subsection{Necessity part of Theorem~\ref{interpolationthm}}

Suppose $\Gamma$ is an interpolation sequence for $A(p,q)$. Corollary
\ref{cor:necessityofseparation} shows that $\Gamma$ is uniformly
discrete. To prove $D^+(\Gamma) \leq 1/q$, we will show that $\Gamma$ is
an interpolation sequence for either $A^q$ or $A^p_{\alpha}$ for all
$\alpha$ with $ (1 + \alpha)/p > 1/q$. In the latter case, the known
results for the Bergman spaces will imply that the density is less than
or equal to $1/q$ and then strict equality follows thanks to the
stability of interpolation sequences under small perturbations (see
Remark \ref{rem:stability}).

The interpolation problem is solved by the formula
\begin{equation}
  f(z) = \sum_j a_j g_j(z) \dfrac{(1-|z_j|^2)^{n+s}}{(1-z\bar{z}_j)^s}
\end{equation}
provided the sum converges for the space in question. Here $n=1/p+1/q$,
and $g_j$ are functions in $A(p,q)$ satisfying $g_j(z_j) =
(1-|z_j|^2)^{-n}$, $g_j(z_{j'}) = 0$ for all $j'\neq j$, and $\lVert g_j
\rVert_{A(p,q)} \leq M$ indepent of $j$. ($M$ is the interpolation
constant of $A(p,q)$.)

First, we show that for $s$ sufficiently large, the series above
converges unifomly on each compact set of the unit disk to an analytic
function $f$. This is done similarly to the suffiency case. The proof
now is split into four cases. We verify that $\norm{f}$ is bounded in
$A^q$ in the first two cases and that it is bounded in $A^{p}_{\alpha}$
(when $(1 + \alpha)/p > 1/q$) in the last two. Since the case $p=q$ is
known for Bergman spaces, we do not need to include it.

\paragraph{Case 1:} $q/p < 1, \,q \leq 1$.\\
By inequality \eqref{ineq01},
\begin{align*}
  \norm{ f }_{A^q}^{q}
    &= \int_0^1\int_0^{2\pi}
        \left|
        \sum_j a_j g_j(z)
        \frac{(1-|z_j|^2)^{n+s}}{(1-z\bar{z}_j)^s}
        \right|^q \,d\theta 2r\,dr \\
    &\le \int_0^1\int_0^{2\pi}
        \sum_j |a_j|^q |g_j(z)|^q
        \frac{(1-|z_j|^2)^{(n+s)q}}{|1-z\bar{z}_j|^{sq}}
        \,d\theta 2r\,dr \\
    &= \int_0^1
        \sum_j |a_j|^q (1-|z_j|^2)^{(n+s)q}
        \left[\int_0^{2\pi}\frac{|g_j(z)|^q}{|1-z\bar{z}_j|^{sq}} \,d\theta\right]
        2r\,dr
\end{align*}

Apply Holder's inequality, we have
\begin{multline*}
    \norm{ f }_{A^q}^{q}
    \le
    \int_0^1
    \sum_j |a_j|^q (1-|z_j|^2)^{(n+s)q}
    \left[\int_0^{2\pi}|g_j(z)|^p \,d\theta\right]^{q/p} \times \\
    \times \left[\int_0^{2\pi} \frac{1}{|1-z\bar{z}_j|^{sq(p/q)'}}
    \,d\theta\right]^{1/(p/q)'} 2r\,dr
\end{multline*}
If $s$ is sufficiently large, inequality \eqref{ineq03} then gives us
\begin{align*}
  \norm{ f }_{A^q}^{q}
    &\le \int_0^1
        \sum_j |a_j|^q (1-|z_j|^2)^{(n+s)q}
        \left[\int_0^{2\pi}|g_j(z)|^p \,d\theta\right]^{q/p}
        \frac{1}{(1-r|z_j|)^{q(s+1/p-1/q)}}
        2r\,dr \\
    &= \sum_j |a_j|^q (1-|z_j|^2)^{(n+s)q}
        \int_0^1
        \frac{1}{(1-r|z_j|)^{q(s+1/p-1/q)}}
        \left[\int_0^{2\pi}|g_j(z)|^p \,d\theta\right]^{q/p} 2r\,dr \\
    &\le C\sum_j |a_j|^q (1-|z_j|^2)^{(n+s)q}
        \frac{1}{(1-|z_j|^2)^{q(s+1/p-1/q)}}
        \int_0^1\left[\int_0^{2\pi}|g_j(z)|^p \,d\theta\right]^{q/p} 2r\,dr \\
    &\le C\sum_j |a_j|^q (1-|z_j|^2)^2
        \norm{ g_j }_{A(p,q)}^q \\
    &\le C\sum_j (1-|z_j|^2)^2 |a_j|^q < \infty.
\end{align*}

\paragraph{Case 2:} $q/p < 1, \,q > 1$.\\
Let $x_1,x_2,y_1,y_2$ be real numbers satisfying $x_1+y_1=n+s,
x_2+y_2=s$. Then by Holder's inequality and inequality \eqref{ineq05},
we have
\begin{align*}
  \norm{ f }_{A^q}^{q}
    &= \int_0^1\int_0^{2\pi}
        \left|
        \sum_j a_j g_j(z)
        \frac{(1-|z_j|^2)^{n+s}}{(1-z\bar{z}_j)^s}
        \right|^q \,d\theta 2r\,dr \\
    &\le C\int_0^1\int_0^{2\pi}
        \left[
        \sum_j |a_j|^q |g_j(z)|^q
        \frac{(1-|z_j|^2)^{qx_1}}{|1-z\bar{z}_j|^{qx_2}}
        \right]
        \left[
        \sum_j\frac{(1-|z_j|^2)^{q'y_1}}{|1-z\bar{z}_j|^{q'y_2}}
        \right]^{q/q'} \,d\theta 2r\,dr \\
    &\le C\int_0^1\int_0^{2\pi}
        \left[
        \sum_j |a_j|^q |g_j(z)|^q
        \frac{(1-|z_j|^2)^{qx_1}}{|1-z\bar{z}_j|^{qx_2}}
        \right]
        (1-r^2)^{q(y_1-y_2)}
        \,d\theta 2r\,dr \\
    &= C\int_0^1
        \sum_j |a_j|^q
        (1-r^2)^{q(y_1-y_2)}
        (1-|z_j|^2)^{qx_1}
        \left[\int_0^{2\pi}\frac{|g_j(z)|^q}{|1-z\bar{z}_j|^{qx_2}}\, d\theta\right]
        \,dr
\end{align*}
Holder's inequality again gives us
\begin{multline*}
    \norm{ f }_{A^q}^{q}
    \le
    C\int_0^1
    \sum_j |a_j|^q
    (1-r^2)^{q(y_1-y_2)}
    (1-|z_j|^2)^{qx_1}
    \left(
    \int_0^{2\pi} |g_j(z)|^p \,d\theta
    \right)^{q/p} \\
    \times
    \left(
    \int_0^{2\pi}
    \frac{1}{|1-z\bar{z}_j|^{qx_2(p/q)'}}\,
    d\theta
    \right)^{1/(p/q)'}
    2r\,dr
\end{multline*}
By inequality \eqref{ineq02}, we have
\begin{multline*}
  \norm{ f }_{A^q}^{q}
    \le
    C\int_0^1
    \sum_j |a_j|^q
    (1-|z_j|^2)^{qx_1}
    (1-r^2)^{q(y_1-y_2)}
    \frac{1}{(1-r|z_j|)^{q(x_2+1/p-1/q)}} \times \\
    \times \left(
    \int_0^{2\pi} |g_j(z)|^p\, d\theta
    \right)^{q/p}
    2r\,dr.
\end{multline*}
Since $x_1-x_2+y_1-y_2=n$, if we choose $x_1=2/q$ then $y_1-y_2=x_2+1/p-1/q$ and thus,
\begin{equation*}
  (1-|z_j|^2)^{qx_1}
    \frac{(1-r^2)^{q(y_1-y_2)}}{(1-r|z_j|)^{q(x_2+1/p-1/q)}}
    \le
    C(1-|z_j|^2)^2
    \frac{(1-r^2)^{q(y_1-y_2)}}{(1-r)^{q(y_1-y_2)}}
    \le
    C(1-|z_j|^2)^2.
\end{equation*}
It follows that
\begin{equation*}
  \norm{ f }_{A^q}^{q}
    \le
    C\sup_j\left\{\norm{ g_j}_{p,q}^{q}\right\}
    \sum_j (1-|z_j|^2)^2 |a_j|^q
    <\infty.
\end{equation*}

\paragraph{Case 3:} $q/p > 1, \,p \leq 1$.\\
Given $\alpha$ with $(1 + \alpha)/p > 1/q$,
\begin{equation*}
\begin{split}
  \norm{ f }_{A_{\alpha}^p}^p
    & =
    \int_{\ID } |f(z)|^p (1-|z|^2)^{\alpha} \,dA(z) \\
    & =
    \int_0^1 \int_0^{2\pi}
    \left|
    \sum_k a_k g_k(z)
    \frac{(1-|z_k|^2)^{n+s}}{(1-z\bar{z}_k)^s}
    \right|^p
    (1-r^2)^{\alpha}
    \,d\theta 2r\,dr
\end{split}
\end{equation*}
By inequality \eqref{ineq01}, we have
\begin{equation*}
\begin{split}
  \norm{ f }_{A_\alpha^p}^p
    & \le \int_0^1 \int_0^{2\pi}
        \sum_k |a_k|^p |g_k(z)|^p
        \frac{(1-|z_k|^2)^{(n+s)p}}{|1-z\bar{z}_k|^{sp}}
        (1-r^2)^{\alpha}
        \,d\theta 2r\,dr \\
    & = \sum_k |a_k|^p
        (1-|z_k|^2)^{(n+s)p}
        \int_0^1 (1-r^2)^{\alpha}
        \left[
        \int_0^{2\pi} \frac{|g_k(z)|^p}{|1-z\bar{z}_k|^{sp}}\, d\theta
        \right] 2r\,dr \\
    & \le \sum_k |a_k|^p
        (1-|z_k|^2)^{(n+s)p}
        \int_0^1
        \frac{(1-r^2)^{\alpha}}{(1-r|z_k|)^{sp}}
        \left[
        \int_0^{2\pi} |g_k(z)|^p\, d\theta
        \right] 2r\,dr
\end{split}
\end{equation*}
Holder's inequality then gives us
\begin{multline*}
  \norm{ f }_{A_{\alpha}^p}^p
    \le
    \sum_k |a_k|^p
    (1-|z_k|^2)^{(n+s)p}
    \left[
    \int_0^1
    \left(
    \int_0^{2\pi} |g_k(z)|^p\, d\theta
    \right)^{q/p} 2r\,dr
    \right]^{p/q} \times \\
    \times \left[
    \int_0^1
    \frac{(1-r^2)^{\alpha(q/p)'}}{(1-r|z_k|)^{sp(q/p)'}} 2r\,dr
    \right]^{1/(q/p)'}
\end{multline*}
From inequality \eqref{ineq03}, if $s$ is sufficiently large, we get
\begin{equation*}
\begin{split}
  \norm{ f }_{A_{\alpha}^p}^p
    & \le
    C\sum_k |a_k|^p
    (1-|z_k|^2)^{(n+s)p}
    \norm{ g_k }_{A(p,q)}^p
    (1-|z_k|^2)^{\alpha-sp+1/(q/p)'} \\
    & \le
    C\sum_k (1-|z_k|^2)^{2+\alpha} |a_k|^p < \infty.
\end{split}
\end{equation*}

\paragraph{Case 4:} $q/p > 1, \,p > 1$.\\
The norm of $f$ (if it exists) in the weighted Bergman space
$A^{p}_{\alpha}$ is
\begin{equation*}\begin{split}
  \norm{ f }_{A^{p}_{\alpha}}^p
    & = \int_{\ID }
        |f(z)|^p (1-|z|^2)^{\alpha}
        \,dA(z) \\
    & = \int_0^1 (1-r^2)^{\alpha}
        \int_0^{2\pi}
        \left|
        \sum_k a_k g_k(z)
        \frac{(1-|z_k|^2)^{n+s}}{(1-z\bar{z}_k)^s}
        \right|^p
        \,d\theta\, 2r\,dr
\end{split}
\end{equation*}
Let $x_1,y_1,x_2,y_2$ be real numbers satisfying $x_1+y_1=n+s,
x_2+y_2=s$. Then by Holder's inequality, we have
\begin{multline*}
  \norm{ f }_{A_{\alpha}^p}^p
    \leq
    \int_0^1 (1-r^2)^{\alpha}
    \int_0^{2\pi}
    \left[
    \sum_k |a_k|^p |g_k(z)|^p
    \frac{(1-|z_k|^2)^{px_1}} {|1-z\bar{z}_k|^{px_2}}
    \right] \times \\
    \times \left[
    \sum_k
    \frac{(1-|z_k|^2)^{p'y_1}} {|1-z\bar{z}_k|^{p'y_2}}
    \right]^{p/p'}
    \,d\theta 2r\,dr
\end{multline*}
In light of inequality \eqref{ineq05}, this implies that
\begin{multline*}
  \norm{ f }_{A_{\alpha}^p}^p\\
    \shoveleft{\quad\le C\int_0^1 \int_0^{2\pi}
    \left[
    \sum_k |a_k|^p |g_k(z)|^p
    \frac{(1-|z_k|^2)^{px_1}} {|1-z\bar{z}_k|^{px_2}}
    \right]
    (1-r^2)^{p\left(y_1-y_2+\alpha/p\right)}
    \,d\theta 2r\,dr} \\
    \shoveleft{ \quad = C\sum_k |a_k|^p (1-|z_k|^2)^{px_1}
    \int_0^1 (1-r^2)^{p\left(y_1-y_2+\alpha/p\right)}
    \left( \int_0^{2\pi}
    \frac{|g_k(z)|^p} {|1-z\bar{z}_k|^{px_2}}
    \,d\theta
    \right) } \\
    \shoveleft{ \quad \le C\sum_k |a_k|^p (1-|z_k|^2)^{px_1}
    \int_0^1 \frac{(1-r^2)^{p\left(y_1-y_2+\alpha/p\right)}} {(1-r|z_k|)^{px_2}}
    \left(
    \int_0^{2\pi} |g_k(z)|^p \,d\theta
    \right)
    2r\,dr } \\
    \shoveleft{ \quad \le C\sum_k |a_k|^p (1-|z_k|^2)^{px_1}
    \left[
    \int_0^1 \left( \int_0^{2\pi} |g_k(z)|^p \,d\theta \right)^{q/p}
    2r\,dr
    \right]^{p/q} \times } \\
    \times
    \left[
    \int_0^1
    \frac{(1-r^2)^{p\left(y_1-y_2+\alpha/p\right)(q/p)'}} {(1-r|z_k|)^{px_2(q/p)'}}\,
    2r\,dr
    \right]^{1/(q/p)'}
\end{multline*}
Inequality \eqref{ineq03} then gives us
\begin{align*}
  \norm{ f }_{A_{\alpha}^p}^p
    & \le C\sum_k |a_k|^p \norm{ g_k }_{A(p,q)}^p
      (1-|z_k|^2)^{px_1 + p\left(y_1-y_2-x_2+\alpha/p \right)+1/(q/p)'} \\
    & \le C\sum_k (1-|z_k|^2)^{2+\alpha} |a_k|^p < \infty.
\end{align*}
Finally, we choose $x_1,x_2,y_2,y_2$ such that they satisfy the
conditions of the inequalities above.

This completes the proof of the necessity part of Theorem \ref{interpolationthm}.

\section{Sampling sequences for mixed-norm spaces}\label{sec:sampling}

Interpolation problems often go together with sampling problems. For a
function space $(A, \norm{.}_A)$ on $\Omega$, a sequence of distinct
points $\Gamma = (z_m) \subset \Omega$ and a sequence space $(X,
\norm{.}_X)$, $\Gamma$ is said to be a sampling sequence for $(A, X)$ if
there exist positive constants $K_1, K_2$ such that for all $f \in A$
\begin{equation}\label{samplinginequalities}
  K_1\norm{f}_A \leq \norm{(f(z_m))}_X \leq K_2\norm{f}_A
\end{equation}

The sampling problem for mixed-norm spaces is to characterize sampling
sequences for the pair $(A(p,q), l^{p,q}(\Gamma))$.
%
%
Our method is based on the paper \cite{luecking2}. There it is shown
that a sequence $\Gamma$ is sampling for $A^q$ if and only if there
exists $r < q$ such that every limit set under a sequence of M\"obius
transformations of $\Gamma$ is a set of uniqueness for $A^r$. Thus, the
density condition that is known to be characteristic of sampling in
$A^q$ must be equivalent to this limit condition. We will show the
following theorem.
\pagebreak

\begin{theorem}
The following are equivalent
\begin{enumerate}
    \item $\Gamma$ is sampling for $A(p,q)$.
    \item $\Gamma$ is sampling for $A^q$.
    \item $\Gamma$ is sampling for $A^p_\alpha$ where $(1+\alpha)/p =
        1/q$.
    \item $\Gamma$ contains a uniformly discrete subsequence $\Gamma'$
        with $D^-(\Gamma') > 1/q$.
\end{enumerate}
\end{theorem}

The equivalence of the last three conditions is known work of K.~Seip
\cite{seip} and A.~Schuster \cite{schusterarticle}. We will prove that
the first condition is equivalent to the second when $q<p$ and to the
third when $p<q$.

But to start, we need a mixed norm version of \cite[Lemma 3.4]{luecking2}.

\begin{lemma}\label{lower-means}
Let $f \in A(p,q)$ and let $r < \min(p,q)$ and $\alpha > r/q - 1$.
For any $z \in \ID$ we have.
\begin{equation}\label{Berezin}
    |f(z)|^r
    \le \frac{\alpha + 1}{\pi} \int_{\ID} |f(w)|^r
        \frac{(1 - |z|^2)^{2 + \alpha}(1 - |w|^2)^\alpha}
             {|1 - \bar w z|^{4 + 2\alpha}}\,dA(w)\,.
\end{equation}
Moreover, there is a constant $C$ depending only on $p$, $q$ and $r$
such that if $B_\epsilon = B_\epsilon(f)$ is the set of points where
\begin{equation}\label{smallmeans}
    |f(z)|^r \le \epsilon \int_{\ID} |f(w)|^r \frac { (1 - |z|^2)^{2 +
        \alpha}(1-|w|^2)^\alpha } { | 1 - \bar w z |^{4 + 2\alpha} }
        \,dA(w)\,.
\end{equation}
\tu(we think of $B_\epsilon$ as the set of `bad' points\/\tu) then
\begin{equation}
    \norm{f\chi_{B_\eps}}_{L(p,q)} \le C\epsilon \norm{f}_{L(p,q)}\,.
\end{equation}
Therefore, $\epsilon > 0$ may be chosen independent of $f$ so that if
$G_\epsilon = \ID \setminus B_\epsilon$ is the set of `good' points
for $f$ then
\begin{equation}\label{largeintegral}
    \norm{f\chi_{G_\eps}}_{L(p,q)} \ge (1/2) \norm{f}_{L(p,q)}\,.
\end{equation}
\end{lemma}

Note that the inequality $\alpha > r/q-1$ will allow us to take $\alpha \le
p/q - 1$ when $p \le q$ and $\alpha < 0$ when $q\le p$.

\begin{proof}
The proof is exactly as in \cite{luecking2}, except we need to show that
the integral operator $T$ with kernel
\begin{equation*}
    K_\alpha(z,w) = \frac { (1 - |z|^2)^{2 + \alpha}(1-|w|^2)^\alpha}
        {| 1 - \bar w z |^{4 + 2\alpha} } =
        \frac{(1-r^2)^{2+\alpha}(1-\rho^2)^\alpha}{|1-r\rho
        e^{i(\theta-t)}|^{4+2\alpha}}
\end{equation*}
(where $z=re^{i\theta}$ and $w=\rho e^{it}$) is bounded from
$L(p/r,q/r)$ to itself.

This is essentially well-known, but we include a proof for convenience.
To simplify the notation, let us show that this operator is bounded on
$L(p,q)$, when $p>1$ and $q>1$ and $\alpha > 1/q - 1$. Then the proof
will apply to $L(p/r,q/r)$ with $\alpha > r/q - 1$.

Given a function $g(w)=g(\rho e^{it})\in L(p,q)$,
\begin{equation*}
    \norm{Tg}_{L^p(d\theta)} \le \int_0^1 \norm{K_\alpha *
    g}_{L^p(d\theta)} 2\rho\,d\rho \le \int_0^1 \norm{K_\alpha}_{L^1(dt)}
    \norm{g}_{L^p(dt)} \,d\rho
\end{equation*}
in which the convolution is taken in the angle variables. We use
part~2 of Lemma~\ref{lem:inequalities} (with $M = 4+2\alpha$, which is
greater than $1$ because $\alpha > -1$) to obtain
\begin{equation*}
  \norm{K_\alpha}_{L^1(dt)} \le
    C\frac{(1-r^2)^{2+\alpha}(1-\rho^2)^\alpha} {|1-r\rho|^{3+2\alpha}}
\end{equation*}
Since $\norm{g}_{L^p(dt)}$ belongs to $L^q (2r\,dr)$, it now suffices to
show that the function in the inequality above defines a bounded
integral operator on $L^q (2r\,dr)$. This is a consequence of the Schur
method (Lemma~\ref{lem:Schur} below) and the inequalities
\begin{equation}\label{eq:Schurinequalities}
\begin{aligned}
  \int_0^1 \frac{(1-r^2)^{2+\alpha}(1-\rho^2)^{\alpha-\eps q'}}
    {|1-r\rho|^{3+2\alpha}} 2\rho\,d\rho
    & \le C (1-r^2)^{-\eps q'}\\
  \int_0^1 \frac{(1-r^2)^{2+\alpha-\eps q}(1-\rho^2)^{\alpha}}
    {|1-r\rho|^{3+2\alpha}} 2r \,dr
    & \le C (1-\rho^2)^{-\eps q}
\end{aligned}
\end{equation}
These both follow from part~3 of Lemma~\ref{lem:inequalities} if $\eps$
is chosen so that $2+2\alpha > \alpha - \eps q' > -1$ and
$2+2\alpha 2 + \alpha - \eps q > -1$. If we solve for $\eps$ and see
that we must have simultaneously
\begin{equation*}
        \frac{-2-\alpha}{q'} < \eps < \frac{1+\alpha}{q'}
        \quad\text{and}\quad
        \frac{-\alpha}{q}    < \eps < \frac{3+\alpha}{q}
\end{equation*}
Such an $\eps$ exists if we have
\begin{equation*}
    \max\left( \frac{-2-\alpha}{q'}, \frac{-\alpha}{q} \right) <
        \min\left( \frac{1+\alpha}{q'}, \frac{3+\alpha}{q} \right)
\end{equation*}
Of the four inequalities that this leads to, two are equivalent to
$\alpha > -3/2$ and the others are respectively equivalent to $\alpha >
-2 - 1/q$ and $\alpha > 1/q - 1$. All of these follow from the last, and
that was one of the assumptions.
\end{proof}

The Schur method for establishing boundedness of integral operators on
$L^p$ spaces can be found in \cite{forellirudin}. Here we use the
following form:

\begin{lemma}\label{lem:Schur}
Let $k(x,y)$ be a nonnegative measurable kernel on the product space
$(X\times Y, \mu\otimes \nu) $ where $\mu$ and $\nu$ are
$\sigma$-finite measures. Let $q > 1$ and assume there exist real constants $C_1$
and $C_2$ and positive functions $h_1(x)$ and $h_2(y)$ such that
\begin{align*}
  \int_X k(x,y)h_1(x)^{q'} \,d\mu(x)
    &\le C_1 h_2(y)^{q'}&&\text{for all } y\in Y,\\
  \int_Y k(x,y)h_2(y)^{q}  \,d\mu(x)
    &\le C_2 h_1(x)^{q} &&\text{for all } x\in X.
\end{align*}
Then the integral operator $K$ defined by $Kf(y) = \int k(x,y)f(x)
\,d\mu(x)$ is bounded from $L^q(X,\mu)$ to $L^q(Y,\nu)$ and $\norm{K}
\le C_1^{1/q'}C_2^{1/q}$.
\end{lemma}

The equations~\eqref{eq:Schurinequalities} establish the hypotheses of
Lemma~\ref{lem:Schur} for the functions $h(z) = k(z) = (1 -
\rho^2)^{-\eps}$, thus finishing the proof of Lemma~\ref{lower-means}.

One half of the sampling requirement for $\Gamma = \{ z_n \}$ is an upper
estimate:
\begin{equation*}
  \| (f(z_n)) \|_{l^{p,q}} \le \| f \|_{L(p,q)}, \quad f\in A(p,q).
\end{equation*}
A necessary and sufficient condition for this is that the number of
points of the sequence in $Q_{j,k}$ is bounded above independent of $j$
and $k$. We will say such a sequence has bounded density and also call
it a \emph{Carleson sequence}. If $\mu$ is sum of unit masses at each
point of $\Gamma$, then $\mu$ satisfies $\mu(Q_{j,k}) \le C$ with $C$
independent of $j$ and $k$. We will use $\C$ for the set of measures
satisfying this inequality for some constant $C$ depending on the
measure.

If $\phi$ is a M\"obius transformation of the disk, write $\mu_\phi$ for
the measure defined by $\mu_\phi(E) = \mu(\phi^{-1}(E))$. We say a
sequence of measures $\mu_n$ \emph{converges weakly} to $\mu$ if $\int h
\,d\mu_n \to \int h \,d\mu$ for all continuous $h$ with compact support
in $\ID$. Let $\W_\mu$ denote the set of all weak limits of measures of
the form $\mu_{\phi_n}$ for sequences $\{ \phi_n \}$ of M\"obius
transformations of $\ID$.

The main result of \cite{luecking2} is that a measure $\mu$ is sampling
for the weighted Bergman space $A^p_\alpha$ if and only there exists
$r < p$ such that the support of every measure in $\W_\mu$ is a set
of uniqueness for $A^{r}_\alpha$. When applied to a sum of point masses on
a sequence $\Gamma$, this condition must be equivalent to
$D^{-}(\Gamma') > (1+\alpha)/p$ for some uniformly discrete subsequence
$\Gamma'$ of $\Gamma$. In particular, if the lower uniform density of
such a $\Gamma'$ is greater than $1/q$, then $\Gamma$ is a sampling
sequence for both $A^q$ and $A^p_{p/q-1}$.

We also need the following lemma from \cite{luecking2} (Lemma 3.7).

\begin{lemma}\label{lueckinginequality}
Let $0 < r < \infty$ and $\alpha > 0$. Let $\epsilon > 0$ be given, and
define
\begin{equation*}
\U_\epsilon = \left\{ f \in A^r_\alpha : \norm{f}_{A^r_\alpha} \le 1
\text{ and } |f(0)| > \epsilon \right\}.
\end{equation*}
Assume that $\mu\in \C$ is such that the support of every measure in
$W_\mu$ is a set of uniqueness for $A^r_\alpha$. Then there is $\delta >
0$ such that $\int \left| f \right|^r (1 - |z|^2)^{\alpha + 2}
\,d\mu_\phi > \delta$ for all $\phi\in\M$ and all $f \in \U_\epsilon$.
\end{lemma}

When $\mu$ is the sum of unit point masses on $\Gamma=(z_n)$ then this
lemma says that if $0$ is one of the `good' points from
Lemma~\ref{lower-means}, that is
\begin{equation*}
  |f(0)|^r > \eps{\norm{f}_{L^r_\alpha}}
\end{equation*}
then there exists a $\delta > 0$ such that for all M\"obius
transformations $\varphi$ we have
\begin{equation*}
  \sum_{n=1}^\infty |f(\varphi(z_n))|^r (1 - |\varphi(z_n)|^2)^{\alpha +
  2} > \delta \norm {f}_{L^r_\alpha}^r > \delta |f(0)|^r
\end{equation*}

We can now prove the sufficiency of the density condition

\begin{theorem}
If there exists a uniformly discrete subsequence $\Gamma'$ of $\Gamma$
such that $D^{-}(\Gamma') > 1/q$ then $\Gamma$ is a sampling sequence
for $A^{p,q}$.
\end{theorem}

\begin{proof}
The case $p=q$ is the known result for $A^q$ and so we divide the proof
into two cases, $p<q$ and $p>q$. We start with the first. Throughout the
proof, let $\alpha = p/q-1$ so that $\alpha>-1$ and $(1+\alpha)/p =
1/q$. Also let $\mu = \sum_{z_n\in \Gamma} \delta_{z_n}$, the sum of
point masses for points in $\Gamma$.

Thus $\Gamma$ is a sampling sequence for $A^p_\alpha$ and so the measure
$\mu$ satisfies the conditions of \cite{luecking2}. In particular,
Lemma~\ref{lueckinginequality} holds for this $\mu$. We have seen that
this implies if $0$ belongs to the good set $G_\eps$ from Lemma~\ref{lower-means}
\begin{equation*}
   |f(0)|^r \le C \int_\ID |f(z)|^r (1-|z|^2)2+\alpha \,d\mu_\phi
\end{equation*}
for some $r < p$ and for for all M\"obius transformations $\phi$. As in
\cite{luecking2} we can consider compositions of $f$ with M\"obius
transformations and conclude that if $\zeta\in G_\eps$
\begin{equation*}
  |f(\zeta)|^r \le C \int \left| f(z) \right|^r \left| \phi_\zeta'(z) \right|^{2 +
    \alpha}(1 - |z|^2)^{2 + \alpha} \,d\mu(z)
\end{equation*}
Expanding the expression $\left| \phi_\zeta'(z) \right|$, this gives
\begin{equation*}
  |f(\zeta)|^r \le C \int | f(z) |^r
    \frac{(1-|\zeta|^2)^{2+\alpha}(1 - |z|^2)^{2 + \alpha}}
    {|1-\bar\zeta z|^{4+2\alpha}} \,d\mu(z)
\end{equation*}

Let us use $\K$ to represent the operator of integration against
\begin{equation*}
    K(\zeta,z) = \frac{(1-|\zeta|^2)^{2+\alpha}(1 - |z|^2)^{2 + \alpha}}
    {|1-\bar\zeta z|^{4+2\alpha}} \,d\mu(z)
\end{equation*}
We claim that $\K$ maps $l^{p/r,q/r}(\Gamma)$ boundedly into
$L^{p,q}$, viewing a sequence in $l^{p/r,q/r}(\Gamma)$ as a function
$f(z)$ on the measure space $(\Gamma, \mu)$. Then we will have
\begin{align*}
 \norm{f}_{A^{p,q}}
    &= \norm{|f|^r}_{L^{p/r,q/r}}\\
    &\le 2 \norm{|f|^r\chi_{G_\eps}}_{L^{p/r,q/r}}\\
    &\le 2C \K((|f(z_m)|^r))                         \\
    &\le 2C\norm{\K}\norm{(|f(z_m)|^r)}_{l^{p/r,q/r}}  \\
    &=   2C\norm{\K}\norm{(f(z_m))}_{l^{p,q}}
\end{align*}
Thus we will have obtained the sampling inequality and completed the
proof of sufficiency.

The proof that $\K$ is bounded is almost identical to the proof of the
boundedness of a similarly defined operator in Lemma~\ref{lower-means}.
The difference between the two operators is that the latter involves
integration against $dA(z)$ whereas the former replaces this with the
measure $(1-|z|^2)^2\,d\mu(z)$. However, because the kernel satisfies
$K(\zeta,z) < CK(\zeta,w)$ for pairs $z,w\in Q_{j,k}$, with $C$
independent of $\zeta$, $(j, k)$, Integration with respect to
$(1-|z|^2)^2\,d\mu(z)$ can be estimated in terms of integration with
respect to $dA(z)$ because $(1-|z|^2)^2d\mu(z)$ is a Carleson measure.
Details are omitted.

The case $q < p$ proceeds similarly and is even simpler because we can
take $\alpha = 0$. What is important is the the exponent $r$ must be
chosen to make both $p/r$ and $q/r$ greater than 1.
\end{proof}

Turning to the necessity we have the following.

\begin{theorem}
If $\Gamma$ is a sampling sequence for $A^{p,q}$ then it contains a
uniformly discrete subsequence $\Gamma'$ satisfying $D^{-}(\Gamma') > 1/q$
\end{theorem}

I will simply sketch the proof. We need the following variant of
Theorem~\ref{thm:stability}:

\begin{lemma}\label{lem:stability}
For $0<p,q<\infty$, let $\Gamma = (u_m)$ be a sampling sequence
for $A(p,q)$ and $(u'_m)$ be another sequence in $\ID $. There
exists $\delta > 0$ such that if
\begin{equation*}
  \rho(u_m,u'_m) < \delta \quad \text{for all } m
\end{equation*}
then $\Gamma' = (u'_m)$ is also a sampling sequence for $A(p,q)$.
\end{lemma}

The proof is omitted. It is very much the same as that for
Theorem~\ref{thm:stability}. A crucial point of that proof is that the
sequence $\Gamma$ is uniformly discrete. That is stronger than necessary
for this lemma: all that is needed is that the sequence be Carleson.

As a first step toward the proof of necessity, we need the following
mixed-norm version of Theorem~3.9 of \cite{luecking2}.

\begin{lemma}\label{perturbation}
Let $\Z$ be a zero sequence for $A^{p,q}$ and assume that
$\Z$ is a Carleson sequence. Let $0 < \gamma < 1$ and suppose there
exists another set $\Z'$ and a one-to-one correspondance $\sigma\colon \Z
\to \Z'$ such that $1 - |\sigma(a)|^2 = \gamma (1 - |a|^2)$ and
$\rho(a,\sigma(a))$ is bounded on $\Z$. Then $\Z'$ is a zero set for
$A^{p/\gamma, q/\gamma}$.
\end{lemma}

The proof of this is almost identical to that of Theorem~3.9 of
\cite{luecking2}. Instead of Theorem~3.8 of \cite{luecking2}, taken from
\cite{Lue96}, we can make use of this similar characterization of zero
sets of the mixed-norm space, also from \cite{Lue96}:

\begin{lemma}
$\Z$ is a zero sequence for $A^{p,q}$
if and only if there exists a harmonic function $h$ in $\ID$ such that
the function $\exp \left[ k_\Z(\zeta) - h(z) \right]$ belongs to
$L^{p,q}$.
\end{lemma}

See either \cite{luecking2} or \cite{Lue96} for the definition of
$k_\Z$.

The next thing we need is the appropriate relationship between mixed
norm spaces and related Bergman spaces.

\begin{lemma}
If $p<q$ let $\alpha = p/q-1$. Then $A^p_\alpha \subset A^{p,q}$ and for
all positive $\lambda < 1$, $A^{p,q}\subset A^{\lambda p}_\alpha$.

If $q < p$ then $A^{p,q} \subset A^q$ and for all positive $\lambda <
1$, $A^q \subset A^{\lambda p,\lambda q}$.
\end{lemma}

If $p<q$ we can deduce from this is the following: $\Gamma$ is a
set of uniqueness for $A^r_\alpha$, for some $r<p$, if and only if it is
a set of uniqueness for $A^{\lambda p,\lambda q}$, for some $\lambda<1$.

If $q<p$ we can deduce: $\Gamma$ is a set of uniqueness for $A^r$, for
some $r<q$, if and only if it is a set of uniqueness for $A^{\lambda
p,\lambda q}$, for some $\lambda<1$.

The main result of \cite{luecking2} then implies that the density
criterion is equivalent to the following condition: there exist
$\lambda<1$ such that the support of every measure in $W_\mu$ is a set
of uniqueness for $A^{\lambda p,\lambda q}$. So our proof of necessity
requires the following:

\begin{theorem}
Suppose for every $\lambda < 1$ there exists a measure in $W_\mu$ whose
support is a zero set for $A^{\lambda p,\lambda q}$. Then $\Gamma$ is
not a sampling sequence for $A^{p,q}$.
\end{theorem}

Finally, the proof of this theorem is essentially identical to that of
the corresponding result in \cite{luecking2}: Theorem~5.1,
(a)${}\Rightarrow{}$(b).

\providecommand{\bysame}{\leavevmode\hbox to3em{\hrulefill}\thinspace}
\providecommand{\MR}{\relax\ifhmode\unskip\space\fi MR }
\providecommand{\MRhref}[2]{%
  \href{http://www.ams.org/mathscinet-getitem?mr=#1}{#2}
}
\providecommand{\href}[2]{#2}

\end{document}